\documentclass[11pt,a4paper,fleqn]{article}
\RequirePackage{amsthm,amsmath,natbib}
\RequirePackage{amssymb,amstext}
\RequirePackage{hypernat}
\usepackage[pdfpagemode=UseOutlines ,plainpages=false
,hypertexnames=false ,pdfpagelabels ,hyperindex=true,colorlinks=true]{hyperref}
	\makeatletter
	\Hy@breaklinkstrue
	\makeatother
\usepackage{color}
\definecolor{darkred}{rgb}{0,0,0}
\definecolor{darkgreen}{rgb}{0,0.5,0}
\definecolor{darkblue}{rgb}{0,0,0.5}
\hypersetup{colorlinks ,linkcolor=black ,filecolor=darkgreen
,urlcolor=black ,citecolor=black}
\renewcommand{\cite}{\citet}
\RequirePackage{bbm}
\renewcommand{\leq}{\leqslant}
\renewcommand{\geq}{\geqslant}
\renewcommand{\epsilon}{\varepsilon}
\newcommand{\set}[1]{{\left\lbrace #1\right\rbrace }}	
\newcommand{\vect}[1]{{\left( #1\right) }}	
\newcommand{\vectp}[1]{{\left( #1\right)_{\hspace*{-1ex}+}}}	
\newcommand{\So}{\phi}

\newcommand{\hSo}{\widehat{\So}}

\newcommand{\Sow}{\beta}
\newcommand{\Sor}{r}
\newcommand{\Soclass}{\cF}
\newcommand{\cSowr}{\Soclass_{\Sow}^\Sor}
\newcommand{\cSow}{\Soclass_{\Sow}}
\newcommand{\Op}{\Gamma}
\newcommand{\fOp}{\fou{\Op}}
\newcommand{\hOp}{\widehat{\Op}}
\newcommand{\fhOp}{\fou{\widehat{\Op}}}

\newcommand{\Opw}{\gamma}
\newcommand{\Opd}{d}
\newcommand{\Opclass}{\mathcal{G}}
\newcommand{\cOpwd}{\Opclass_{\Opw}^\Opd}
\newcommand{\Opother}{\mathrm{T}}
\newcommand{\gf}{g}
\newcommand{\hgf}{\widehat{\gf}}
\newcommand{\Li}{\ell}
\newcommand{\fLi}{\fou{\Li}}
\newcommand{\hLi}{\widehat{\Li}}
\newcommand{\tLi}{\widetilde{\Li}}
\newcommand{\mstar}{m^*}
\newcommand{\mstarn}{{\mstar_n}}

\newcommand{\moptn}{{m^*_n}}
\newcommand{\mloptn}{{m^\diamond_n}}
\newcommand{\umloptn}{\underline{m^\diamond_n}}

\newcommand{\fou}[1]{[#1]}

\newcommand{\bas}{\psi}

\newcommand{\proj}{\Pi}
\newcommand{\Sspace}{\Hz}
\newcommand{\Hspace}{\Hz}
\newcommand{\galsol}{\So_{m}}
\newcommand{\hOpset}[1]{\Omega_{#1}}
\newcommand{\hOpsetmn}{\hOpset{m,n}}
\newcommand{\Xiset}[1]{\mho_{#1}}
\newcommand{\Xisetmn}{\Xiset{m,n}}
\newcommand{\aset}[1]{\mathcal{A}_{#1}}
\newcommand{\asetn}{\aset{n}}
\newcommand{\bset}[1]{\mathcal{B}_{#1}}
\newcommand{\bsetn}{\bset{n}}
\newcommand{\cset}[1]{\mathcal{C}_{#1}}
\newcommand{\csetn}{\cset{n}}
\newcommand{\eset}[1]{\mathcal{E}_{#1}}
\newcommand{\esetn}{\eset{n}}
\def\pen{\mathop{\rm p}\nolimits}
\def\hpen{\mathop{\rm \widehat{p}}\nolimits}
\def\bias{\mathop{\rm b}\nolimits}
\def\contr{\mathop{\rm \kappa}\nolimits}
\newcommand{\Mu}{M^-}
\newcommand{\Mun}{\Mu_n}
\newcommand{\Mo}{M^+}
\newcommand{\Mon}{\Mo_n}

\newcommand{\Men}{\widehat{M}_n}
\newcommand{\Mln}{{M^\Li_n}}

\newcommand{\Mfunc}{M}
\newcommand{\gauss}[1]{\lfloor #1\rfloor}
\newcommand{\floor}[1]{\lfloor #1\rfloor}
\newcommand{\quart}{{1/4}}
\newcommand{\nlogn}{{n(1+\log n)^{-1}}}
\newcommand{\lognn}{{(1+\log n)n^{-1}}}
\newcommand{\lU}{U}
\newcommand{\lV}{V}
\newcommand{\lW}{W}
\newcommand{\lZ}{Z}
\def\Ex{{\mathbb E}}

\def\argmin{\mathop{\rm arg \; min}\limits}
\def\tr{\mathop{\rm tr}\nolimits}
\def\Diag{\mathop{\rm \nabla}\nolimits}
\newcommand{\skalarV}[1]{\langle #1\rangle}
\DeclareMathOperator{\norm}{\lVert\cdot\rVert}  
\newcommand{\normV}[1]{\lVert#1\rVert}   
\newcommand{\absV}[1]{\lvert#1\rvert}  
\newcommand{\HskalarV}[1]{\langle #1\rangle_{\Hz}}   
\DeclareMathOperator{\Hskalar}{\langle\cdot,\cdot\rangle_{\Hz}} 
\DeclareMathOperator{\Hnorm}{\lVert\cdot\rVert_{\Hz}} 
\newcommand{\HnormV}[1]{\lVert#1\rVert_{\Hz}}  
\newcommand{\mnormV}[1]{\lVert#1\rVert_{s}}  
 \newcommand{\Hz}{{\mathbb H}}
\newcommand{\Nz}{{\mathbb N}}
\newcommand{\Rz}{{\mathbb R}}
\newcommand{\Zz}{{\mathbb Z}}
\def\1{\mathop{\mathbbm 1}\nolimits}
\newcommand{\Id}{{[\mathrm{I}]}}
\newcommand{\cE}{{\cal E}}
\newcommand{\cF}{{\cal F}}
\newcommand{\cG}{{\cal G}}
\newcommand{\cR}{{\cal R}}
 \newcommand{\uk}{{\underline{k}}}
\newcommand{\um}{{\underline{m}}}
\newcommand{\tm}{\widetilde{m}}
\newcommand{\whm}{\widehat{m}}

 \newcommand{\hV}{\widehat{V}}
\newcommand{\hsigma}{\widehat{\sigma}}
\newcommand{\htau}{\widehat{\tau}}
\newcommand{\sfs}{q}
\definecolor{dgreen}{rgb}{0,0,0}
\definecolor{dblue}{rgb}{0,0,0}
\definecolor{dred}{rgb}{0,0.0,0}
\definecolor{dgold}{rgb}{0,0,0}
\definecolor{dvio}{rgb}{0,0,0}
\definecolor{gray}{rgb}{0,0,0}

\newtheoremstyle{mysc}
  {3pt}
  {3pt}
  {\it}
  {}
  {\color{black}\sc}
  {.}
  {.5em}
  {}
\newtheoremstyle{myex}
  {10pt}
  {10pt}
  {\rm}
  {}
  {\color{black}\sc}
  {.}
  {.5em}
  {}
\theoremstyle{mysc}\newtheorem{prop}{Proposition}[section]
\theoremstyle{mysc}\newtheorem{assumption}{Assumption}[section]
\theoremstyle{mysc}
\theoremstyle{mysc}\newtheorem{theo}[prop]{Theorem}
\theoremstyle{mysc}
\theoremstyle{mysc}\newtheorem{lem}[prop]{Lemma}
\theoremstyle{myex}\newtheorem{rem}{Remark}[section]
\theoremstyle{myex}
\theoremstyle{myex}
\numberwithin{equation}{section}

\author{{\sc Jan Johannes}\thanks{Institut de statistique, biostatistique et sciences actuarielles (ISBA), Voie du Roman Pays 20, B-1348 Louvain-la-Neuve,
    Belgium, e-mail: \url{{jan.johannes|rudolf.schenk}@uclouvain.be}} \and {\sc Rudolf Schenk$^*$}}

\title{\bf Adaptive estimation of linear functionals \\ in functional linear models}

\begin{document}
\date{Universit{\'e} catholique de Louvain}
\maketitle

\begin{abstract}We consider the estimation of the value of a linear functional of the slope parameter in functional linear regression, where scalar responses are modeled in
  dependence of random functions. In \cite{JohannesSchenk2010} it has been shown that a plug-in estimator based on dimension reduction and additional thresholding can attain minimax
  optimal rates of convergence up to a constant. However, this estimation procedure requires an optimal choice of a tuning parameter  with regard to certain characteristics of
  the slope function and the covariance operator associated with the functional regressor. As these are unknown in practice, we investigate a fully data-driven choice of the tuning parameter
 based on a combination of model selection and Lepski's method, which is inspired by the recent work of \cite{GoldenshlugerLepski2011}. The tuning parameter is  selected as the
 minimizer of a stochastic  penalized contrast function imitating Lepski's method among a random collection of admissible values. We show that this adaptive procedure attains the lower bound for the minimax risk up to a logarithmic factor over a wide range of classes of slope functions and
  covariance operators. In particular, our theory covers point-wise estimation as well as the estimation of local averages of the slope parameter.
 \end{abstract} {\footnotesize
\begin{tabbing}
\noindent \emph{Keywords:} \=Adaptation, Linear functional, Lepski's method, 
Model selection, Linear Galerkin projection,\\
\>  Minimax-theory, Point-wise estimation, Local average estimation, Sobolev space.\\[.2ex]
\noindent\emph{AMS 2000 subject classifications:} Primary 62J05; secondary 62G05, 62G20
\end{tabbing}
This work was supported by the IAP research network no.\ P6/03 of the
Belgian Government (Belgian Science Policy) and by the ``Fonds
Sp\'eciaux de Recherche'' from the Universit\'e catholique de Louvain.
}

\section{Introduction}
The functional linear model with scalar response describes the relationship between a real random variable $Y$ and the variation of a functional regressor $X$. Usually, the random
function $X$ is assumed to be square integrable or more generally  to take its  values in a separable Hilbert space 
 $\Hspace$ with the inner product $\Hskalar $ and associated norm $\Hnorm$. For convenient notations  we assume that the regressor $X$ is  centered in the  sense that for all $h \in \Hspace$ the real valued random variable $\HskalarV{X,h}$  
  has mean zero. The linear relationship between  $Y$ and $X$ is expressed by the equation 
 \begin{equation}\label{intro:model}Y=\HskalarV{\So,X} +\sigma\epsilon,\quad\sigma>0,
 \end{equation}
 with the unknown slope parameter $\So\in
  \Hspace$  and a real-valued, centered and standardized error term $\epsilon$. 
The objective of this paper is the fully data-driven estimation of the value of a known linear functional of the 
slope $\So$ based on an independent and identically distributed (i.i.d.) sample of $(Y,X)$ of size $n$.

The   estimation of the value  of a  
linear functional  
  offers a general framework for naturally arising related estimation problems, such as estimating the value of $\So$
  - or of one of  its derivatives - 
 at a given point or estimating the average of $\So$ over a subinterval of its domain.

There is extensive literature available on the topic of
 non-parametric estimation of the value of a linear functional from Gaussian white
 noise observations   (in case of direct observations see \cite{Speckman1979}, \cite{Li1982} or \cite{IbragimovHasminskii1984}, while in case
 of indirect observations we refer to \cite{DonohoLow1992}, \cite{Donoho1994} or \cite{GoldPere2000} and references therein).  In the situation of a functional linear model as
 considered  in \eqref{intro:model}, which does  in general  not lead to Gaussian white noise observations, \cite{JohannesSchenk2010} have  investigated  the  minimax optimal performance of a plug-in estimator
for the value   of a linear functional $\Li$ evaluated at $\So$. For this purpose 
 the slope $\So$ is  replaced in $\Li(\So)$ by  a suitable estimator
 $\hSo_\moptn$ depending  on a tuning parameter $\moptn\in\Nz$. However their choice of the tuning parameter is not data-driven.   In the present paper we develop a data-driven
 selection procedure which features comparable minimax-optimal properties.
 
The non-parametric estimation of the slope function $\So$ 
has been an issue of growing interest in the recent literature and a variety of such estimators  have been studied.
 For example, \cite{Bosq2000},
\cite{CardotMasSarda2007} or \cite{MullerStadtmuller2005} 
analyze a functional principal components regression, while a penalized least squares approach combined with projection
onto some basis (such as splines) is examined in \cite{RamsayDalzell1991}, \cite{EilersMarx1996}, \cite{CardotFerratySarda2003}, \cite{HallHorowitz2007} or
  \cite{CrambesKneipSarda2007}. \cite{CardotJohannes2008} investigate a linear Galerkin approach
coming from the inverse problem community (c.f. \cite{EfromovichKoltchinskii2001} and \cite{HoffmannReiss04}).
The resulting thresholded projection estimator $\hSo_\moptn$ is used by \cite{JohannesSchenk2010} in  their plug-in estimation procedure $\hLi_{\moptn}:=\Li(\hSo_\moptn)$ for   the value $\Li(\So)$ of a linear functional evaluated at $\So$.

 It has been shown in \cite{JohannesSchenk2010} that the attainable rate of convergence of the plug-in estimator is basically determined by the \textit{a priori}
conditions on the solution $\So$ and the covariance operator $\Op$  associated with the regressor $X$ (defined below). These conditions are expressed  in the form $\So\in\cF$  and $\Op\in \cG$, for suitably chosen classes $\cF\subseteq \Hspace$ and $\cG$; we postpone their formal introduction along with their
 interpretation  to Section \ref{sec:mar}.
Moreover, the accuracy of any estimator $ \tLi$
of the value $\Li(\So)$ has been assessed  by its maximal mean squared error  with respect to these classes, that is 
\begin{equation*} 
 \cR^\Li[\tLi; \cF,\cG]:=\sup_{\So\in\cF}\sup_{\Op\in\cG}
\Ex|\tLi-\Li(\So)|^2.
\end{equation*}
The main purpose 
of \cite{JohannesSchenk2010} has been to derive  a lower bound
\[
\cR^\Li_*[n^{-1}; \cF,\cG ]\leq \inf\nolimits_{\tLi}\cR^{\Li}[\tLi;\cF,\cG]
,
\]
where the infimum is taken over all  estimators $\tLi$, and 
 to prove that the estimator
$\hLi_\moptn$ satisfies
\[\cR^\Li[\hLi_\moptn;\cF,\cG]\leq C\cdot 
\cR^\Li_*[n^{-1}; \cF,\cG ]
,\quad \text{with }0<C<\infty,\]
 for a variety of classes $\cF$ and $\cG$. In other words it has been shown that  $\cR^\Li_*[n^{-1}; \cF,\cG ]$ is the minimax-optimal rate attained by the estimator $\hLi_\moptn$. The optimal performance  of the estimator depends crucially on the  choice  $m^*_n$ of the tuning parameter, which in turn, relies  strongly on \textit{a priori} knowledge of
 the    sets $\cF$\ and $\cG$. However, this information is widely inaccessible in practice. 

The aim of the present paper consists in proposing a fully data-driven selection procedure for the tuning parameter. Our selection method combines model selection
(c.f. \cite{BarronBirgeMassart1999} and its detailed discussion in \cite{Massart07}) and Lepski's method (c.f. \cite{Lepski1990} and its recent review in \cite{Mathe2006}). It is inspired by the
recent work of \cite{GoldenshlugerLepski2011} who consider data-driven bandwidth selection  in kernel density estimation. We choose the appropriate tuning parameter $\whm$ as the minimizer of a stochastic  penalized
contrast function imitating Lepski's method among a random collection of admissible values. Furthermore,  we  show that the maximal  risk of the resulting estimator $\hLi_{\whm}$ satisfies
\[
\cR^\Li[\hLi_{\whm};\cF,\cG]\leq C\cdot
\cR^\Li_*[(1+\log n) n^{-1};\cF,\cG]\quad\text{for } 0<C<\infty,
\]
 for a variety of classes $\cF$ and $\cG$. The upper bound in the last display features a logarithmic factor  when compared to the minimax rate of convergence
 $\cR^\Li_*[n^{-1}; \cF,\cG ]$ which possibly results in a deterioration of the rate.  Therefore, the completely data-driven  estimator is optimal or nearly optimal in the minimax sense simultaneously over
a variety of both solution sets $\cF$\ and  classes of operators $\cG$. We call such estimation procedures adaptive. The appearance of the logarithmic factor within the rate is a
known fact in the context of local estimation (c.f. \cite{LLP2008} who consider model selection given  direct  Gaussian observations).  \cite{BrownLow1996}  show
that it is unavoidable in the context of non-parametric Gaussian regression and, hence it is widely considered as an acceptable price for adaptation.  This factor is also present in the recent
work of \cite{GoldPere2000} where Lepski's method is applied in the presence of  indirect Gaussian observations. In contrast to this situation  the
operator is not known in advance in functional linear regression and  hence a straightforward application of their results is not obvious. We will show that our proposed
data-driven estimation method attains the minimax-rates up to a logarithmic factor for a variety of a classes of both slope functions and covariance operators.

The paper is organized as follows: in Section \ref{sec:mar} we introduce the adaptive estimation procedure and review the available minimax theory as presented in
\cite{JohannesSchenk2010}. In Section \ref{sec:urb} we present the key arguments of the proof of an upper  risk bound for the adaptive estimator, while more technical aspects of
the proof are deferred to the Appendix.  We discuss the examples of point-wise and local average estimation in Section \ref{sec:ex}. 
\section{Methodology and review}\label{sec:mar}
We suppose that the regressor $X$ has a finite second
moment, i.e., $\Ex\HnormV{X}^2<\infty$, and that $X$ is uncorrelated to the
random error $\epsilon$ in the sense that $\Ex{[\epsilon\HskalarV{X,h}]}=0$ for
all $h\in\Hspace$, as usually assumed in this context, see  for example 
\cite{Bosq2000}, \cite{CardotFerratySarda2003} or
\cite{CardotMasSarda2007}.  Multiplying both sides in (\ref{intro:model}) by $\HskalarV{X,h}$
and taking the expectation leads to  the normal equation
 \begin{equation}\label{intro:e2:2}\HskalarV{\gf,h}:=\Ex[Y\HskalarV{X,h}]
 =\Ex[\HskalarV{\So,X}\HskalarV{X,h}]=:\HskalarV{\Op \So,h},\quad \forall h \in \Hspace,
 \end{equation}
 where $\gf$ belongs to $\Hspace$ and $\Op$ denotes the covariance operator associated with the random function $X$.
In what follows we  assume that there exists a unique  solution $\So\in \Hspace$ of equation (\ref{intro:e2:2}), i.e., that $\Op$ is strictly positive and that  its range contains $\gf$  (for a detailed discussion we refer to \cite{CardotFerratySarda2003}). 
 Obviously, these conditions are sufficient for the identification of the value $\Li(\So)$.
  Since the estimation of $\So$ involves an
inversion of the covariance operator $\Op$  it is called an \emph{inverse} problem. Moreover, due to the finite second moment of the regressor $X$, the associated covariance
 operator $\Op$ is  nuclear, i.e., its trace is finite.   Therefore, the  reconstruction of
 $\So$ leads to an \emph{ill-posed inverse problem} (with the additional difficulty that $\Op$ is unknown and has to be estimated). 
In the following  we assume that the joint distribution of   the regressor and error term is Gaussian, more precisely, we suppose that for any finite set $\{h_1,\ldots,h_{k-1}\}\subset\Hspace$ the vector 
 $( \HskalarV{X,h_1}\,\ldots,\HskalarV{X,h_{k-1}},\epsilon)$ follows a $k$-dimensional multivariate normal distribution.
\begin{rem}
The assumption of Gaussianity is not  essential for the proof of our main result. This assumption on the distributions of the error and the regressor is
 only used to prove the bounds given in Lemma \ref{app:gauss:l2}. Analogues of the results can be shown at the cost of  longer
proofs under appropriately chosen moment conditions.\hfill$\square$\end{rem}  

\subsection{Adaptive Estimation Procedure}
\paragraph{Introduction of the estimator.}
In order to derive an estimator for the unknown slope function $\So$ we  follow the presentation of \cite{JohannesSchenk2010} and base our reconstruction on the development of $\So$ in an arbitrary orthonormal basis. Here
 and subsequently, we fix a pre-specified orthonormal basis $\{\bas_j\}_{j=1}^{\infty}$   of $\Hspace$ which   does in general not  correspond to the eigenfunctions of the operator
$\Op$ defined in \eqref{intro:e2:2}. We require in the following that the slope function $\So$ belongs to a function class $\Soclass$ containing $\{\bas_j\}_{j= 1}^{\infty}$ and, moreover that $\cF$
is included in the domain  of the linear functional $\Li$. For technical reasons and without loss of generality we assume  that $\Li(\psi_1)=1$ which can always be ensured by reordering
and rescaling, except for the trivial case $\ell\equiv 0$. With respect to this basis, we consider for all $h\in\Hspace$ the development $h=\sum_{j=1}^\infty \fou{h}_j\bas_j$ where the sequence
$\fou{h}:=(\fou{h}_j)_{j\geq1}$   of generalized Fourier coefficients $\fou{h}_j:=\HskalarV{h,\bas_j}$ is square-summable, i.e.,
$\HnormV{h}^2=\sum_{j=1}^\infty\fou{h}_j^2<\infty$. Given a dimension parameter $m\in\Nz$  we have the subspace $\Sspace_m$ - spanned by the basis functions $\{\bas_j\}_{j=1}^m$ - at our disposal  and
we call $\So_{m}\in\Sspace_{m}$ a Galerkin solution of $\gf=\Op\So$, if  $\HnormV{\gf-\Op\So_{m}}\leqslant  \HnormV{\gf-\Op h}$ for all $h\in\Sspace_{m}$.
Since $\Op$ is strictly positive it is easily seen that the Galerkin solution $\So_{m}$ of  $\gf=\Op\So$ exists uniquely.
Let us introduce for any function $h$ the  $m$-dimensional vector of coefficients  $\fou{h}_{\um}:=(\fou{h}_j)_{1\leq j\leq m}$ and for the operator $\Op$  the $(m\times m)$-dimensional matrix $\fOp_{\um}:=(
\HskalarV{\bas_j,\Op\bas_k}
)_{1\leq j,k\leq m}$. 
Then the Galerkin solution $\So_m$  satisfies  $
\fOp_{\um}\fou{\So_m}_{\um}= \fou{\gf}_{\um}$. Since $\Op$ is injective,  the matrix $\fOp_{\um}$ is non-singular for all $m\geq 1$ and therefore the Galerkin solution $\So_m\in\Sspace_m$ is uniquely determined by the vector of coefficients $
\fou{\So_m}_{\um}= \fOp_{\um}^{-1}\fou{\gf}_{\um}$ and  $\fou{\So_m}_j=0$ for $j>m$. In order to derive an estimator  for the  vector $\fou{\So_m}_{\um}$, we replace the unknown quantities $\fou{\gf}_{\um}$ and $\fOp_{\um}$
 by their empirical counterparts and apply additional thresholding. 
We observe that $\fOp_{\um}= \Ex\fou{X}_{\um}\fou{X}_{\um}^t$ and $\fou{\gf}_{\um}=\Ex Y\fou{X}_{\um}$, therefore, given an i.i.d.\ sample $\{(Y_i,X_i)\}_{i=1}^n$ of $(Y,X)$, 
it is natural
to consider the estimators  $\fou{\hgf}_{\um}:=\frac{1}{n}\sum_{i=1}^n Y_{i}\fou{X_{i}}_{\um}$
and $\fhOp_{\um}:=\frac{1}{n}\sum_{i=1}^n[X_i]_{\um}[X_i]_{\um}^t$. Let us denote by $\mnormV{ \fhOp^{-1}_{\um}     }$  the spectral norm of $    \fhOp^{-1}_{\um}  $, i.e., its largest eigenvalue, 
 and define the estimator  $\hSo_m\in\Sspace_m$ by means of the coefficients  $\fou{\hSo_m}_j=0$ for $j>m$  and 
\begin{equation*}
\fou{\hSo_m}_{\um}:=\left\{\begin{array}{ll} 
\fhOp_{\um}^{-1} [\hgf]_{\um}, & \mbox{if $\fhOp_{\um}$ is non-singular and }\mnormV{\fhOp^{-1}_{\um}}\leq n,\\
0&\mbox{otherwise}.
\end{array}\right.
\end{equation*}
Observe that $\Li(\So_m)=(\Li(\bas_1),\dotsc,\Li(\bas_m))\fou{\So_m}_\um=:\fLi_\um^t\fou{\So_m}_\um$ with the slight abuse of notations  $\fLi_\um:=(\fLi_j)_{1\leq j\leq m}$  and generic
elements $\fLi_j:=\Li(\bas_j)$. In \cite{JohannesSchenk2010} it has been shown that the estimator $\hLi_m:=\Li(\hSo_m)$ with optimally chosen dimension parameter $m$   can attain
minimax-optimal rates of convergence. This choice involves certain characteristics of the slope $\So$ and the covariance operator $\Op$ which are unavailable in practice. In the
next paragraph we introduce a fully data-driven selection method for the dimension parameter. 
 \paragraph{Introduction of the adaptive estimation procedure.}
Our selection method is 
inspired by the
recent work of \cite{GoldenshlugerLepski2011} 
and combines the
techniques of model selection  and Lepski's method. We determine the dimension parameter
 among a collection of admissible values
 by minimizing a penalized contrast function. To this end, we define for all $n\geq1$
 the value   $\Mln:=\max\set{1\leq m\leq \gauss{n^{1/4}}: \fLi_\um^t\fLi_\um\leq n}$ where $\gauss{a}$ denotes as usual the integer
part of $a\in\Rz$ and introduce the random integer 
\begin{equation}\label{estimator:def:Men}
  \Men:=\min\set{2\leq m\leq \Mln: \mnormV{\fhOp_\um^{-1}} (\fLi_\um^t\fLi_\um)> n(1+\log n)^{-1}}-1.
\end{equation}
Furthermore, we define a stochastic penalty sequence $\hpen:=(\hpen_m)_{1\leq m\leq \Men}$ by
\begin{equation*}
  \hpen_m:=700 \vect{\frac{2}{n}\sum_{i=1}^nY_i^2 + 2\fou{\hgf}_\um^t\fhOp_\um^{-1}\fou{\hgf}_\um} \cdot \max_{1\leq k\leq m} \fLi_\uk^t \fhOp_\uk^{-1}\fLi_\uk\cdot \frac{(1+\log
    n)}{n}.
\end{equation*}
The random integer $\Men$ and the stochastic penalty $\hpen_m$ are used to define a  contrast by 
\begin{equation*}
  \contr_m:=\max_{m\leq k\leq\Men}\set{|\hLi_k-\hLi_m|^2  -\hpen_k}.
\end{equation*}
For a subset $A\subset\Nz$ and  a sequence $(a_m)_{m\geq1}$ with minimal value in $A$  we set  $\argmin\nolimits_{m\in A}\{a_m\}:=\min\{m:a_m\leq a_{m'},\forall m'\in A\}$  and  select the dimension parameter
\begin{equation}
  \label{estimator:def:whm}
  \whm:=\argmin_{1\leq m\leq\Men}\set{\contr_m +\hpen_m}.
\end{equation}
The  estimator of $\Li(\So)$ is now given by $\hLi_{\whm}$ and  we will derive an upper bound for its risk below. By construction the choice of the dimension
parameter and hence the estimator  $\hLi_{\whm}$ rely only on the data and in particular not on the regularity assumptions on the slope and the operator which we formalize
in the next section.
\subsection{Review of minimax theory}\label{sec:romt}
We express our \textit{a priori} knowledge about the unknown slope parameter and covariance operator in the form $\So\in\Soclass$ and
$\Op\in\Opclass$. The class $\Soclass$  reflects  information on the solution $\So$, e.g., its level of smoothness, whereas the assumption $\Op\in\Opclass$ typically results in conditions on the decay of the
eigenvalues of the operator $\Op$. The following construction of the classes $\Soclass$ and $\Opclass$ will be flexible enough to characterize, in particular,  differentiable  or analytic slope functions and  allows us to discuss both a polynomial and
exponential decay of the covariance operator's eigenvalues.

\paragraph{Assumptions and notations.} With respect to the basis $\{\bas_j\}_{j=1}^\infty$ and given a strictly positive sequence of weights $(w_j)_{j\geq 1}$, or $w$ for short, we define the  weighted norm $\norm_{w}$ by 
$\normV{h}_{w}^2:= \sum_{j=1}^\infty w_j\fou{h}_j^2$ for $h\in\Hspace$. 
Throughout the rest of the paper  let $\Sow$ be a  non-decreasing sequence of weights    with $\Sow_1=1$ such that slope parameter $\So$ belongs to the ellipsoid 
\begin{equation*}
  \cSowr:=\set{ h\in\Hspace: \normV{h}^2_\Sow\leq \Sor}\quad\text{with radius  $\Sor>0$.} 
\end{equation*}
In order to guarantee that  $\cSowr$ is contained in the domain of the linear functional $\Li$ and that $\Li(h)=\sum_{j\geq1} \fLi_j\fou{h}_j$ for all
$h\in\cSowr$ with
$\fLi_j=\Li(\bas_j)$, $j\geq1$,  it is sufficient that  $\sum_{j\geq1} \fLi_j^2\Sow_j^{-1}<\infty$. We may emphasize
that we  neither impose that the sequence $\fLi=(\fLi_j)_{j\geq1}$ tends to zero nor that it is square summable. However, if it is square summable then $\Hspace$ is the domain of $\Li$. Moreover,
$\fLi$ coincides with the sequence of generalized Fourier coefficients of the representer of $\Li$ given by Riesz's theorem.\\

As usual in the context of ill-posed inverse problems, we link the mapping properties of the covariance operator $\Op$ and the regularity conditions on $\So$.  To this end, we consider the sequence $(\skalarV{\Op
  \bas_j,\bas_j})_{j\geq1}=: (\fOp_{jj})_{j\geq 1}$. Since $\Op$ is nuclear, this sequence is summable and hence vanishes as $j$ tends to infinity.  In what follows we impose restrictions on the decay of this
sequence. Let  $\cG$ denote  the set of all strictly positive nuclear operators defined on $\Hspace$.  We suppose that there exists  a strictly positive, summable
sequence of weights $\Opw$ with $\Opw_1=1$   such that $ \Op$ belongs to the subset 
\begin{equation*}
\cOpwd:=\Bigl\{ T\in\cG:\quad   \Opd^{-2}\normV{h}_{\Opw^2}^2\leqslant \HnormV{Th}^2\leqslant {\Opd^2}\, \normV{h}_{\Opw^2}^2,\quad \forall h \in \Hspace\Bigr\}\quad\mbox{with }\Opd\geq1
\end{equation*}
where  we understand here and subsequently arithmetic operations on a sequence of real numbers   component-wise, e.g., we write $\Opw^2$ for $(\Opw_j^2)_{j\geq1}$.
Notice  that  for   $\Op\in\cOpwd$  it follows that
 $d^{-1}\Opw_j\leq \fOp_{jj}\leq d\Opw_j$.  Moreover,   if $\lambda$ denotes its  sequence  of  eigenvalues, then  $   {\Opd^{-1}} \Opw_j \leq \lambda_j \leq {\Opd}
 \Opw_j$ which justifies the condition $\sum_{j=1}^\infty\Opw_j<\infty$. Let us summarize the previous conditions:
 \begin{assumption}\label{minimax:ass:reg}
The sequences $1/\Sow$ and $\Opw$ are monotonically decreasing with limit zero and  $\Sow_1=\Opw_1=1$ such that  $\sum_{j\geq1}\fLi_j^2\Sow_j^{-1}<\infty$ and $\sum_{j\geq1}\Opw_j<\infty$.
\end{assumption}
\paragraph{Illustration.}
\noindent We illustrate the last assumption for typical choices of the
sequences $\Sow$,   
$\Opw$ and $\fLi$. Consider  $\fLi^2_j=|j|^{-2s}$  and: 
 \begin{enumerate}
  \item[\textit{(pp)}]
 $\Sow_j=|j|^{2p}$,   $\Opw_j=|j|^{-2a}$
 with $p>0$, $ a>1/2$ and $s>1/2-p$;
 \item[\textit{(pe)}] 
 $\Sow_j=|j|^{2p}$, $\Opw_j =\exp(-|j|^{2a}+1)$
 with $p>0$, $a>0$ and $s>1/2-p$;
 \item[\textit{(ep)}] 
 $\Sow_j=\exp(|j|^{2p}-1)$,
 $\Opw_j=|j|^{-2a}$ 
with  $p>0$, $a>1/2$
 and  $s\in\Rz$;
 \end{enumerate}
then Assumption \ref{minimax:ass:reg} holds true in all cases.

 \paragraph{Minimax theory reviewed.}
  \cite{JohannesSchenk2010} have derived a  lower bound for the minimax risk $\inf_{\tLi}\cR^\Li[\tLi;\cSowr,\cOpwd]$
and have shown that the proposed estimator $\hLi_{m}$ can attain this lower bound up to constant provided that  the dimension parameter is chosen appropriately. 
 In order to formulate the minimax rate  below let us define for $m\geq 1$ and $x\in(0,1]$
 \begin{multline*}
\cR^\Li_m[x;\cSowr,\cOpwd]:= \max\set{\sum_{j>m} \frac{\fLi_j^2}{\Sow_j},  \max\Bigl(\frac{\Opw_{m}}{\Sow_{m}},x\Bigr)\sum_{j=1}^{m}\frac{\fLi_j^2}{\Opw_j}}\\\text{ and }\;\cR^\Li_*[x;\cSowr,\cOpwd]:=\min_{m\geq1}\cR^\Li_m[x;\cSowr,\cOpwd].
 \end{multline*}
With this notation the lower bound,   when considering  an i.i.d.\ sample of  size $n$, is basically a multiple of   $\cR^\Li_*[n^{-1};\cSowr,\cOpwd]$. To be more precise, if  we define $\mstarn:=\argmin\nolimits_{m\geq1}  \cR^\Li_m[n^{-1}; \cSowr,\cOpwd]$ and if 
 Assumption \ref{minimax:ass:reg} and  $\inf_{n\geq1}\min\big(\frac{\Sow_{\mstarn}}{n\Opw_{\mstarn}},\frac{n\Opw_{\mstarn}}{\Sow_{\mstarn}}\big)>0$ are satisfied then there exists a constant $C>0$
depending only on the classes and $\sigma^2$  such that we have for all $n\geq 1$
\begin{equation*}
  \inf_{\tLi}\cR^\Li[\tLi;\cSowr,\cOpwd]\geq C \cdot \cR^\Li_*[n^{-1};\cSowr,\cOpwd].
\end{equation*}
On the other hand it is shown in  \cite{JohannesSchenk2010}  that  $\cR^\Li_*[n^{-1};\cSowr,\cOpwd]$  provides up to a constant  an upper bound for the maximal risk of the proposed estimator $\hLi_{\mstarn}$. More precisely, if we assume in  addition $\sup_{m\geq
   1}m^3\Opw_m\Sow_m^{-1}<\infty$ then there exists  a constant $C>0$
depending only on the classes and $\sigma^2$ such that  we have for all $n\geq 1$
\begin{equation*} 
 \cR^\Li[\hLi_{\mstarn};\cSowr,\cOpwd]\leq C\cdot  \cR^\Li_*[n^{-1};\cSowr,\cOpwd].
\end{equation*}
Consequently 
 the rate $ \cR^\Li_*[n^{-1};\cSowr,\cOpwd]$ is optimal and  $\hLi_{m^*}$   is  minimax-optimal.

 \paragraph{Illustration continued.}
For the configurations defined below Assumption \ref{minimax:ass:reg} 
  the estimator   $\hLi_{\mstarn}$     with  dimension parameter $\mstarn$  as
  given below is minimax optimal under the following conditions. The minimax optimal rate of convergence is determined by the orders of $ \cR^\Li_*[n^{-1};\cSowr,\cOpwd] $. Here and subsequently, we use for two strictly positive sequences $(x_n)_{n\geq1},(y_n)_{n\geq1}$ the notation $x_n\asymp y_n$, if $(x_n/y_n)_{n\geq1}$ is bounded away both from zero and infinity.
\begin{enumerate}
 \item[\textit{(pp)}]
 If  $p>0$, $a>1/2$ and $p+a\geq 3/2$
 then $\mstarn\asymp n^{1/(2p+2a)}$
 and if
  $s>1/2-p$, then\\
 $ \cR^\Li[\hLi_{\mstarn};\cSowr,\cOpwd]\asymp
\begin{cases}
 n^{-(2p+2s-1)/(2p+2a)}, &\text{if $s-a<1/2$}\\
 n^{-1}\log n, &\text{if $s-a=1/2$}\\
 n^{-1}, &\text{if $s-a>1/2$}.
\end{cases}
$ 

\item[\textit{(pe)}] If
 $p>0$ and $a>0$,
 then  $\mstarn\asymp \log(n(\log n)^{-p/a})^{1/(2a)}$
 and if 
 $s>1/2-p$, then 
$ \cR^\Li[\hLi_{\mstarn};\cSowr,\cOpwd]\asymp
(\log n)^{-(2p+2s-1)/(2a)}$.

\item[\textit{(ep)}] 
If 
$p>0$,
 $a>1/2$ and $s\in\Rz$ then $\mstarn\asymp \log(n(\log n) ^{-a/p})^{1/(2p)}$ and

  $ \cR^\Li[\hLi_{\mstarn};\cSowr,\cOpwd]\asymp
\begin{cases}
 n^{-1}(\log n)^{(2a-2s+1)/(2p)}, &\text{if $s-a<1/2$}\\
 n^{-1}\log(\log n) , &\text{if $s-a=1/2$}\\
 n^{-1}, &\text{if $s-a>1/2$}.
\end{cases}
$
\end{enumerate}
\section{Upper risk bound for the adaptive estimator}\label{sec:urb}
The fully adaptive estimator $\hLi_{\whm}$ of $\Li(\So)$ relies on the choice of a random dimension parameter $\whm$ which  does not involve any knowledge about the classes $\cSowr$ and $\cOpwd$. The main result of
this paper consists in an upper bound for the maximal risk $\cR^\Li[\hLi_{\whm};\cSowr,\cOpwd]$  given by  the following theorem. We
present the main arguments of its proof in this section whereas the more technical aspects are deferred to the appendix. We close this section by illustrating and discussing the result. 
\begin{theo}\label{adaptive:t1} Assume an i.i.d.\ sample of $(Y,X)$ of size $n$ obeying \eqref{intro:model} and let the joint distribution of the random function $X$ and the error $\epsilon$ be normal.
Consider  sequences $\Sow$ and $\Opw$ satisfying Assumption \ref{minimax:ass:reg}. 
Define $\mloptn:=\argmin\nolimits_{m\geq1}\cR^\Li_m[\lognn;\cSowr,\cOpwd]$ and suppose that $\Opw_{\mloptn}^{-1}\fLi_{\umloptn}^t\fLi_{\umloptn}=o(\nlogn)$ as $n\to\infty$ then there exists a constant  $C>0$ depending on the classes $\cSowr$ and $\cOpwd$, the linear functional $\Li$, and  $\sigma^2$ only such that
\begin{equation*}
\cR^\Li[\hLi_{\whm};\cSowr,\cOpwd]\leq C\cdot \cR^\Li_*[\lognn;\cSowr,\cOpwd],\quad\mbox{ for all }n\geq1.
\end{equation*}
\end{theo}
\begin{rem}
The last assertion  states that the data-driven estimator can attain the minimax-rates up to a logarithmic factor for a variety of classes $\cSowr$ and $\cOpwd$. In this sense the
estimator adapts to both the slope function and the covariance operator. This result is derived under the additional condition, $\Opw_{\mloptn}^{-1}\fLi_{\umloptn}^t\fLi_{\umloptn}=o(\nlogn)$ as
$n\to\infty$,  which naturally holds true in the illustrations.\hfill$\square$
\end{rem}
We begin our reasoning by giving a preparatory lemma which constitutes a central step in the following arguments. 
\begin{lem}\label{adaptive:l1}Let $(\phi_k)_{k\geq1}$ be an arbitrary sequence in $\Hspace$ and $\bias:=(\bias_m)_{m\geq1}$  the  sequence of approximation errors  $\bias_m=\sup_{m\leq
    k} \absV{\Li(\So_k-\So)}$ associated with $\Li(\So)$. Consider an  arbitrary  sequence of penalties $\pen:=(\pen_m)_{m\geq 1}$, an upper bound  $M\in\Nz$, and  the sequence $\contr=(\contr_m)_{m\geq1}$ of  contrasts given by $\contr_m:=\max_{m\leq k\leq
  M}\set{\absV{\hLi_k-\hLi_m}^2 -\pen_k }$. If the subsequence  $(\pen_1,\dotsc,\pen_M)$ is non-decreasing, then 
 we have for the  selected model $ \tm:=\argmin\nolimits_{1\leq m\leq M}\set{\contr_m+\pen_m}$ and 
for all $1\leq m\leq M$ that
 \begin{equation}\label{adaptive:l1:e1}
 \absV{\hLi_{\tm}-\Li(\So)}^2 \leq 7\pen_m +78\bias_m^2+42\max_{m\leq k\leq M} \vect{\absV{\hLi_{k}-\Li(\So_k)}^2 -\frac{1}{6}\pen_k}_+
\end{equation}
where $(a)_+=\max(a,0)$.
\end{lem}
\begin{proof}[\noindent\textcolor{darkred}{\sc Proof of Lemma \ref{adaptive:l1}.}]
Since   $(\pen_1,\dotsc,\pen_M)$ is non-decreasing it is  easily verified that
\begin{equation*}
  \contr_m\leq 6\max_{m\leq k\leq M} \vect{\absV{\hLi_{k}-\Li(\So_{k})}^2 -\frac{1}{6}\pen_k}_+ + 12\bias_m^2, \quad \forall\, 1\leq m\leq M,
\end{equation*}
where we use that $2\bias_m\geq \max_{m\leq k\leq M} \absV{\Li(\So_k-\So_m)}$. The last estimate implies the inequality
\begin{equation}\label{pr:adaptive:l1:e1} \absV{\hLi_{m}-\Li(\So)}^2 \leq  \frac{1}{3}\pen_m + 2\bias_m^2+ 2 \max_{m\leq k\leq M} \vect{\absV{\hLi_{k}-\Li(\So_{k})}^2 -\frac{1}{6}\pen_k}_+, \, \forall\, 1\leq m\leq M. \end{equation}
On the other hand,   taking  the definition of $\tm$ into account, it is straightforward to see that
\begin{multline*}
  \absV{\hLi_{\tm}-\Li(\So)}^2\leq 3\set{ \absV{\hLi_{\tm}-\hLi_{\min(m,\tm)})}^2+ \absV{\hLi_{\min(m,\tm)}-\hLi_m}^2+ \absV{\hLi_{m}-\Li(\So)}^2}\\
\hfill\leq 3\set{ \contr_m + \pen_{\tm} +\contr_{\tm}+\pen_m +  \absV{\hLi_{m}-\Li(\So)}^2} 
\leq 6\set{\contr_m +\pen_m}+3  \absV{\hLi_{m}-\Li(\So)}^2.
\end{multline*}
From the last estimates and \eqref{pr:adaptive:l1:e1} we obtain the assertion \eqref{adaptive:l1:e1}, which completes the proof.\end{proof}
The proof of Theorem \ref{adaptive:t1} requires in addition to the previous lemma  two technical propositions which we state now.  
 For $n\geq1$ and  a positive sequence $a:=(a_m)_{m\geq1}$ let us introduce $\Mln:=\max\{1\leq m\leq \floor{n^\quart}:\fLi_{\um}^t\fLi_{\um}\leq n\}$ and 
\begin{equation*}  \Mfunc_{n}(a):=\min\set{2\leq m\leq \Mln: a_m\cdot\fLi_{\um}^t\fLi_{\um} >\nlogn}-1
\end{equation*}
where we set  $\Mfunc_n(a):=\Mln$ if the set is empty.  Observe that  $\Men$ given 
in \eqref{estimator:def:Men} satisfies $\Men=\Mfunc_n(a)$ with $a=(\mnormV{\fhOp_{\um}^{-1}})_{m\geq1}$. Consider for $m\geq 1$
\begin{equation*} \sigma^2_m:=2 \Ex Y^2 +2 [\gf]_{\um}^t\fOp_{\um}^{-1}[\gf]_{\um},\qquad\quad V_m:=\max_{1\leq k\leq m}\fLi_{\uk}^t\fOp_{\uk}^{-1}\fLi_{\uk}
\end{equation*}
and define the penalty term
\begin{equation*}
\pen_{m}:=100\,\sigma_m^2\, V_m\,  \lognn,
\end{equation*}
which are obviously  the theoretical counterparts of the random  objects used in the definition of $\whm$. The proof of the next assertion is deferred to the appendix.
\begin{prop}\label{adaptive:p1}Let the conditions of Theorem \ref{adaptive:t1} hold true and denote by $\So_m\in\Sspace_m$ the Galerkin solution of $\gf=\Op\So$. Define $\Mon:=\Mfunc_n(a)$ with  $a=([4\Opd\Opw_j]^{-1})_{j\geq1}$ then there is a
  constant $C(\Opd)>0$ depending on $\Opd$ only such that for all $n\geq1$
 \begin{multline*}
\sup_{\So\in\cSowr}\sup_{\Op\in\cOpwd} \Ex \set{\max_{1\leq m\leq \Mon} \vect{\absV{\hLi_{m}-\Li(\So_m)}^2 -\frac{\pen_{m}}{6}}_+}
\\\leq \frac{C(\Opd)}{n}(\sigma^2+\Sor) \max\set{(\sum_{j\geq1}\Opw_j)^2,\sum_{j\geq1}\frac{\fLi^2_j}{\Sow_j}}.
 \end{multline*}
\end{prop}
Additionally, let us introduce  for $n\geq1$ the random integer  $\Mun:=\Mfunc_n\big(a\big)$ with the sequence $a=(16\Opd^3\Opw_j^{-1})_{j\geq1}$. In the following we decompose the risk with respect to an
event $\esetn$, and respectively  its complement $\esetn^c$, on which  $\hpen$ and $\Men$ are comparable to their theoretical counterparts.  To be more precise, we define the event
\begin{equation*}\esetn:= \set{\forall\,1\leq m\leq \Mon :   \pen_{m}\leq\hpen_{m}\leq 24
\pen_{m}}\cap\set{\Mun\leq \Men\leq \Mon}
\end{equation*}
 and consider the elementary identity 
  \begin{multline}\label{adaptive:decomp}
\sup_{\So\in\cSowr}\sup_{\Op\in\cOpwd} \Ex \absV{\hLi_{\whm}-\Li(\So)}^2
=\sup_{\So\in\cSowr}\sup_{\Op\in\cOpwd} \Ex \big(\absV{\hLi_{\whm}-\Li(\So)}^2\1_{\esetn}\big) \\+\sup_{\So\in\cSowr}\sup_{\Op\in\cOpwd} \Ex  \big(\absV{\hLi_{\whm}-\Li(\So)}^2\1_{\esetn^c}\big).
\end{multline}
The next proposition states that the second right hand side term is bounded up to a constant by $n^{-1}$ and is hence
negligible.
The proof  is deferred to the appendix.
\begin{prop}\label{adaptive:p2}
Let the conditions of Theorem \ref{adaptive:t1} hold true. If we consider the fully data-driven choice $\whm$ given in \eqref{estimator:def:whm} then there exists a constant  $C(\Opd)>0$   depending on
$\Opd$ only such that for all $n\geq1$
 \begin{equation*}
\sup_{\So\in\cSowr}\sup_{\Op\in\cOpwd} \Ex\big( \absV{\hLi_{\whm}-\Li(\So)}^2\1_{\esetn^c}\big)\leq\frac{ C(\Opd)}{n}\,(\sigma^2+\Sor)\max\set{ \sum_{j\geq1}\Opw_j, \sum_{j\geq1}\frac{\fLi_j^2}{\Sow_j}}.
 \end{equation*}
\end{prop}
We are now in position to prove  Theorem \ref{adaptive:t1}.
\begin{proof}[\noindent\textcolor{darkred}{\sc Proof of Theorem \ref{adaptive:t1}.}]In the following we will denote by $C(\Opd)>0$ a constant depending on $\Opd$ only,  which may
  change from line to line. From  the elementary identity \eqref{adaptive:decomp} and Proposition \ref{adaptive:p2} we derive for all $n\geq1$
  \begin{multline}\label{pr:adaptive:t1:e1}
\sup_{\So\in\cSowr}\sup_{\Op\in\cOpwd} \Ex \absV{\hLi_{\whm}-\Li(\So)}^2\leq\sup_{\So\in\cSowr}\sup_{\Op\in\cOpwd} \Ex \big(\absV{\hLi_{\whm}-\Li(\So)}^2\1_{\esetn}\big)\\ +\frac{C(\Opd)}{n}(\sigma^2+\Sor)\max\set{ \sum_{j\geq1}\Opw_j, \sum_{j\geq1}\frac{\fLi_j^2}{\Sow_j}}.
\end{multline}
We observe that the random subsequence $(\hsigma^2_1,\dotsc,\hsigma^2_{\Men})$,  and hence  $(\hpen_1,\dotsc,\hpen_{\Men})$,  are by construction
non-decreasing. Furthermore, we observe that for all $1\leq m\leq k\leq \Men $ the identity $\HskalarV{\hOp(\hSo_k-\hSo_m),(\hSo_k-\hSo_m)}=[\hgf]_{\uk}^t\fhOp_{\uk}^{-1}[\hgf]_{\uk}-[\hgf]_{\um}^t\fhOp_{\um}^{-1}[\hgf]_{\um}$ holds true. Therefore, it follows
by using that $\hOp$ is positive definite  that $[\hgf]_{\um}^t\fhOp_{\um}^{-1}[\hgf]_{\um}\leq [\hgf]_{\uk}^t\fhOp_{\uk}^{-1}[\hgf]_{\uk}$, and hence $\hsigma^2_m\leq \hsigma_k^2$.
Consequently, Lemma \ref{adaptive:l1}  is applicable for all $1\leq m \leq \Men$ and we obtain
 \begin{equation*}
 \absV{\hLi_{\whm}-\Li(\So)}^2 \leq 7\hpen_m +78\bias_m^2+42\max_{m\leq k\leq \Men} \vect{\absV{\hLi_k-\Li(\So_k)}^2 -\frac{1}{6}\hpen_k}_+.
\end{equation*}
On the event $\esetn$ we  deduce from the last bound that for all $1\leq m\leq \Mun$
\[\absV{\hLi_{\whm}-\Li(\So)}^2\1_{\esetn}\leq 504\pen_m +78\bias_m^2+42{\max_{1\leq m\leq \Mon} \vect{\absV{\hLi_{m}-\Li(\So_m)}^2 -\frac{1}{6}\pen_{m}}_+}.\]
Taking Lemma \ref{app:pre:l1} (v) in the appendix into account it follows for all $n\geq1$
\begin{multline*}
\sup_{\So\in\cSowr}\sup_{\Op\in\cOpwd} \Ex \big(\absV{\hLi_{\whm}-\Li(\So)}^2\1_{\esetn}\big)  \leq C(\Opd)(\sigma^2+\Sor)\min_{1\leq m\leq \Mun}\,\cR^\Li_m\big[\lognn;\cSowr,\cOpwd\big]\\\hfill +\sup_{\So\in\cSowr}\sup_{\Op\in\cOpwd}\Ex\set{\max_{1\leq m\leq \Mon} \vect{\absV{\hLi_{m}-\Li(\So_m)}^2 -\frac{1}{6}\pen_{m}}_+}.
\end{multline*}
Moreover, Proposition \ref{adaptive:p1} and  \eqref{pr:adaptive:t1:e1} imply for all $n\geq1$ that
\begin{multline}\label{pr:adaptive:t1:e2}
\sup_{\So\in\cSowr}\sup_{\Op\in\cOpwd} \Ex \absV{\hLi_{\whm}-\Li(\So)}^2  \leq C(\Opd) (\sigma^2+\Sor)\max\set{ \sum_{j\geq1}\Opw_j, \sum_{j\geq1}\frac{\fLi_j^2}{\Sow_j}}\\ \cdot\min_{1\leq m\leq \Mun}\,\cR^\Li_m\big[\lognn;\cSowr,\cOpwd\big] 
\end{multline}
where we use that $\cR^\Li_m\big[\lognn;\cSowr,\cOpwd\big]\geq n^{-1}$ for all $m\geq 1$. Under the additional condition $\Opw_{\mloptn}^{-1}\fLi_{\umloptn}^t\fLi_{\umloptn}=o(\nlogn)$ it is easily verified that  there exists an integer $n_o$ only depending on the sequences $\Sow$, $\Opw$ and $\fLi$ such that for all $n\geq n_o$ we have $\mloptn\leq \Mun$ and
  \[\min_{1\leq m\leq \Mun}\,\cR^\Li_m[\lognn;\cSowr,\cOpwd]=\cR^\Li_*[\lognn;\cSowr,\cOpwd].\]However,  in case $n<n_o$ we employ that \[\cR^\Li_1[\lognn;\cSowr,\cOpwd]\leq
  \max(1,\lognn)\sum_{j>1}\frac{\fLi_j^2}{\Sow_j}\leq \sum_{j\geq1}\frac{\fLi_j^2}{\Sow_j}\] and consequently we derive the bound 
\[\min_{1\leq m\leq \Mun}\,\cR^\Li_m[\lognn;\cSowr,\cOpwd]\leq n^{-1} n_o \sum_{j\geq1}\frac{\fLi_j^2}{\Sow_j}, \quad \mbox{ for all }n<n_o.\] The combination of both cases  yields for all $n\geq 1$  
 \[\min_{1\leq m\leq \Mun}\,\cR^\Li_m[\lognn;\cSowr,\cOpwd]\leq n_o  \sum_{j\geq1}\frac{\fLi_j^2}{\Sow_j} \cR^\Li_*[\lognn;\cSowr,\cOpwd].\]
As $n_o$ depends only on the sequences $\Sow$, $\Opw$ and $\fLi$, we derive the result of the theorem from the previous display together with \eqref{pr:adaptive:t1:e2}, which completes the proof.
\end{proof}
\begin{rem}\label{Adaptive:R1}Recall that the  estimator $\hLi_{\moptn}$ with optimally chosen dimension parameter $\moptn$  is minimax-optimal, i.e, its maximal risk
$\cR^\Li[\hLi_{\moptn};\cSowr,\cOpwd]$  can be bounded up to a constant by the lower bound $\cR^\Li_*[n^{-1};\cSowr,\cOpwd]$. However, due to Theorem \ref{adaptive:t1} the maximal risk of the fully adaptive estimator 
 is bounded by a multiple of $\cR^\Li_*[\lognn;\cSowr,\cOpwd]$. The appearance of the logarithmic factor within the rate is a known fact in the context of local estimation. It is
 widely considered as an acceptable price for adaptation  (in the context of non-parametric Gaussian regression it is unavoidable as shown in \cite{BrownLow1996}).\hfill$\square$\end{rem}

\paragraph{Illustration continued.}
In the configurations defined below Assumption \ref{minimax:ass:reg} 
the additional condition   $\Opw_{\mloptn}^{-1}\fLi_{\umloptn}^t\fLi_{\umloptn}=o(\nlogn)$
as $n\to\infty$
 is easily verified. Therefore, the maximal risk of the fully adaptive estimator is bounded by a multiple of $\cR^\Li_*[\lognn;\cSowr,\cOpwd]$ due to Theorem \ref{adaptive:t1}.  
In the next assertion we state its order  in the considered cases and we omit  the  straightforward calculations.
\begin{prop} Assume an i.i.d.\ sample of $(Y,X)$ of size $n$ obeying \eqref{intro:model} and let the joint distribution of the random function $X$ and the error $\epsilon$ be normal.
 The obtainable rate of convergence is determined by the orders of\, 
$\cR^\Li_*[\lognn;\cSowr,\cOpwd]$  as given below.
\begin{enumerate}
 \item[\textit{(pp)}]
 If  $p>0$,  $a>1/2$, $p+a\geq 3/2$
 and
  $s>1/2-p$, then\\[1ex]
$\cR^\Li_*[\lognn;\cSowr,\cOpwd]\asymp
\begin{cases}
 \big(n^{-1}\log n\big)^{(2p+2s-1)/(2p+2a)}, &\text{if $s-a<1/2$}\\
  n^{-1}(\log n)^2, &\text{if $s-a=1/2$}\\
 n^{-1}\log n, &\text{if $s-a>1/2$}.
\end{cases}$

\item[\textit{(pe)}] If
 $p>0$,  $a>0$, and if  $s>1/2-p$, then \\[1ex]
$\cR^\Li_*[\lognn;\cSowr,\cOpwd]\asymp
(\log n)^{-(2p+2s-1)/(2a)}$.

\item[\textit{(ep)}] 
If 
$p>0$,
 $a>1/2$ and $s\in\Rz$ then\\[1ex]  
 $ \cR^\Li_*[\lognn;\cSowr,\cOpwd]\asymp
\begin{cases}
n^{-1}(\log n)^{(2p+2a-2s+1)/(2p)} , &\text{if $s-a<1/2$}\\
n^{-1}(\log n)(\log\log n), &\text{if $s-a=1/2$}\\
n^{-1}\log n, &\text{if $s-a>1/2$}.
\end{cases}
$
 \end{enumerate}
\end{prop}
We shall briefly compare these rates with the corresponding minimax optimal rates derived in Section  \ref{sec:romt} above. Surprisingly they coincide in case \textit{(pe)}, and
hence the fully data-driven estimator is minimax-optimal. The rates given in case \textit{(pp)} coincides with the ones that have been obtained by  \cite{GoldPere2000} for  an {\it
  a priori} known operator. In comparison to the  minimax optimal rates the cases  \textit{(pp)} and  \textit{(ep)}  feature  a deterioration of logarithmic order as expected (compare Remark \ref{Adaptive:R1}). 

\section{Examples: point-wise and local average estimation}\label{sec:ex}
Consider $\Hspace=L^2[0,1]$ with its usual norm and inner product and  the trigonometric basis
\begin{equation*}\label{bm:def:trigon:1}
\psi_{1}:\equiv1, \;\psi_{2j}(s):=\sqrt{2}\cos(2\pi j s),\; \psi_{2j+1}(s):=\sqrt{2}\sin(2\pi j s),s\in[0,1],\; j\in\Nz.\end{equation*} 
Recall the typical choices of the sequences $ \Sow$  and $\Opw$ as introduced in the illustrations above.
If  $\Sow_j\asymp |j|^{2p}$ for a positive integer $p$, see  cases \textit{(pp)} and \textit{(pe)}, then  
 the subset $\cSow:=\{h\in\Hspace: \normV{h}_{\Sow}^2<\infty\}$ coincides with the Sobolev space  of $p$-times differential periodic functions
(c.f. \cite{Neubauer1988,Neubauer88}). 
In the case \textit{(ep)}
  it is well-known that for 
$p>1$ every element of
$\cSow$ is 
an analytic function (c.f. \cite{Kawata1972}).
Furthermore we consider a polynomial decay of $\Opw$ with $a>1/2$ in the cases \textit{(pp)} and \textit{(ep)}.
Easy calculus 
shows that the covariance operator $\Op\in\cOpwd$ acts for integer $a$  
    like integrating   $(2a)$-times and is hence called 
{\it finitely smoothing} (c.f. \cite{Natterer84}). 
In the  case \textit{(pe)} we assume   an exponential decay of $\Opw$ and
it is 
 easily  seen that  the range of  $\Op\in \cOpwd$ is a subset of $C^\infty[0,1]$, 
therefore  the operator
is called {\it infinitely smoothing} (c.f. \cite{Mair94}).

\paragraph{Point-wise estimation.}  By  \textit{evaluation   in a given point}
$t_0 \in [0,1]$ we mean the linear functional $\Li_{t_0}$ mapping  $h$ to $h(t_0):=\Li_{t_0}(h)=\sum_{j=1}^\infty [h]_j\bas_j(t_0)$.
In the following we shall assume that the point evaluation is well-defined on
the set of slope parameters $\cSow$ which is obviously implied by  $\sum_{j=1}^\infty[\Li_{t_0}]_j^2\Sow_j^{-1}<\infty$. 
Consequently, the condition  $\sum_{j\geq1}\Sow_j^{-1}<\infty$  is sufficient  to guarantee that the point evaluation is well-defined on $\cSow$.
Obviously, in case  \textit{(ep)} or in other words  for exponentially increasing $\Sow$,  this additional condition is automatically satisfied. However, a polynomial increase, as in the cases  \textit{(pp)} and \textit{(pe)},  
requires the  assumption $p>1/2$. Roughly speaking, this means that the slope parameter has at least to be  continuous. In order to estimate the value
$\So (t_0)$ we consider the plug-in estimator 
\begin{equation*}
\hLi_{t_0}^m
=
\left\{\begin{array}{lcl} 
[\Li_{t_0}]_{\um}^t\fhOp_{\um}^{-1} [\widehat{g}]_{\um}, && \mbox{if $\fhOp_{\um}$ is non-singular and }\mnormV{\fhOp^{-1}_{\um}}\leq  n,\\
0,&&\mbox{otherwise},
\end{array}\right.
\end{equation*}
with $[\Li_{t_0}]_{\um}=(\psi_1(t_0),\dotsc,\psi_m(t_0))^t$. Moreover, we observe that $\hLi_{t_0}^m=\Li_{t_0}(\hSo_m)=\hSo_m(t_0)$.\\[1ex]
{\it Minimax optimal point-wise estimation.} The estimator's maximal mean squared error over
the classes
 $\cSowr$ and $\cOpwd$ is uniformly bounded for all $t_0\in[0,1]$ up to a constant by $\cR^{\Li_{t_0}}_*[n^{-1};\cSowr,\cOpwd]$, i.e., 
$
\sup_{\So  \in \cSowr} \sup_{\Op \in \cOpwd}\Ex|\hSo_{\moptn}(t_0) -\So (t_0)|^2
\leq C\, \cR^{\Li_{t_0}}_*[n^{-1};\cSowr,\cOpwd]$ for some $C>0$, which is  the minimax-optimal rate of convergence (c.f. \cite{JohannesSchenk2010}).\\[1ex]
{\it Illustration continued.} We derive with  $[\Li_{t_0}]^2_j\asymp j^{-2s}$ and $s=0$ in the considered cases :
\begin{itemize}
 \item[\textit{(pp)}] 
If  $p>1/2$, $a>1/2$ and  $p+a\geq3/2$, then $\cR^{\Li_{t_0}}_*[n^{-1};\cSowr,\cOpwd]\asymp n^{-(2p-1)/(2p+2a)}$.
 \item[\textit{(pe)}]  If $p>1/2$ and $a>0$, then $\cR^{\Li_{t_0}}_*[n^{-1};\cSowr,\cOpwd]\asymp (\log n)^{-(2p-1)/2a}$.
 \item[\textit{(ep)}] If $p>0$ and $a>1/2$, then $\cR^{\Li_{t_0}}_*[n^{-1};\cSowr,\cOpwd]\asymp n^{-1}(\log n)^{(2a+1)/2p}$.
\end{itemize}
{\it Adaptive point-wise estimation.}  We select the dimension parameter $\whm$ by minimizing  the penalized contrast function over the collection of admissible values.
  The obtainable rate for the fully data-driven estimator $\hSo_{\whm}(t_0)$   in the three considered cases is given as follows:
\begin{itemize}
 \item[\textit{(pp)}] 
If  $p>1/2$, $a>1/2$ and  $p+a\geq3/2$, then 
$
\cR^{\Li_{t_0}}_*[(1+\log n)n^{-1};\cSowr,\cOpwd]
\asymp (n^{-1}\log n)^{(2p-1)/(2p+2a)}$.
 \item[\textit{(pe)}]  If $p>1/2$ and $a>0$, then 
$\cR^{\Li_{t_0}}_*[(1+\log n)n^{-1};\cSowr,\cOpwd]\asymp (\log n)^{-(2p-1)/(2a)}$.
 \item[\textit{(ep)}] If $p>0$ and $a>1/2$, then 
$\cR^{\Li_{t_0}}_*[(1+\log n)n^{-1};\cSowr,\cOpwd]\asymp n^{-1}(\log n)^{(2p+2a+1)/(2p)}$.
\end{itemize}
The proposed fully data-driven point wise estimator is minimax optimal in case (pe)
which is easily seen by comparing the rates of the adaptive estimator with the corresponding minimax rate.
In the other cases, the rates deviate only by logarithmic factor, as expected.

\paragraph{Point-wise estimation of derivatives.}
It is interesting to note that by  slightly adapting the previously presented procedure we are able to  estimate the value of the $\sfs$-th derivative of $ \So $ at $t_0$. 
Given the exponential basis, which is linked 
to the trigonometric basis 
for $k \in \Zz$ and $t \in [0, 1]$
by the relation
  $\exp(2i \pi kt) = 2^{−1/2}
 (\bas_{2k}(t) + i \bas_{2k+1}(t))$ 
 with $i^2 = −1$. We recall that for $0\leq  \sfs < p$ the $\sfs$-th derivative $\So^{(\sfs)}$
 of $\So$ in a weak sense
 satisfies
\[ \So^{
 (\sfs)}(t_0) =
 \sum_{
 k\in\Zz}
 (2i \pi k)^{\sfs}\exp(2i\pi kt_0)
\Bigl(
\int
^1
_0
 \So(u) 
\exp(−2i\pi ku) 
du
\Bigr)
.\]
 Given a dimension $m \geq 1$, we denote now by $\fou{\hOp}_{\um}$ the $(2m + 1) \times (2m + 1)$ matrix with
 generic elements $\HskalarV{\bas_j
 ,\hOp \bas_k }$, $-m \leq j, k \leq m$ and by $\fou{\hgf}_{\um}$ the $(2m + 1)$ vector with elements
 $\HskalarV{\hgf, \bas_j} $, $-m \leq j \leq m$. Furthermore, we define  for integer $\sfs$   the $(2m + 1)$ vector $\fou{\Li^{(\sfs)}_{t_0}}_{\um}$ with elements
$  
\fou{\Li^{(\sfs)}_{t_0}}_{j}:=(2i \pi j)^{\sfs}\exp(2i\pi j t_0)
$, $-m \leq j \leq m$.
In the following we shall assume that the point evaluation of the $\sfs$-th derivative  is well-defined on
the set of slope parameters $\cSow$ which is  implied by  $\sum_{j\geq1}(j^{2\sfs}\Sow_j^{-1})<\infty$,  since $|[\Li_{t_0}^{(\sfs)}]_j|^2\asymp j^{2\sfs}$.
Obviously, this additional condition is automatically satisfied in case  \textit{(ep)} and requires the  assumption $\sfs<p-1/2$ in the cases  \textit{(pp)} and \textit{(pe)}. 
We consider the estimator of $\So
^{(\sfs)}(t_0)=\Li^{(\sfs)}_{t_0}(\So)$
 given by
\[
  \hSo^{(\sfs)}_m({t_0})=\begin{cases}                         \fou{\Li^{(\sfs)}_{t_0} }_{\um}^t
\fou{\hOp}_{\um}^{-1}
\fou{\hgf}_{\um} &
\text{if $\fhOp_{\um}$ is non-singular and }\mnormV{\fhOp^{-1}_{\um}}\leq  n,\\
0,&\text{otherwise.}
                        \end{cases}
\]
{\it Minimax optimal point-wise estimation of derivatives.} 
 The estimator $\hSo^{(\sfs)}_{\moptn}(t_0)$ with  appropriately chosen dimension  is minimax
 optimal, i.e., 
$
\sup_{\So  \in \cSowr} \sup_{\Op \in \cOpwd}\Ex|\hSo^{(\sfs)}_{\moptn}(t_0) -\So^{(\sfs)} (t_0)|^2
\leq C\, \cR^{\Li_{t_0}^{(\sfs)}}_*[n^{-1};\cSowr,\cOpwd]$ for some $C>0$, where  $\cR^{\Li_{t_0}^{(\sfs)}}_*[n^{-1};\cSowr,\cOpwd]$ is  the minimax-optimal rate of convergence (c.f. \cite{JohannesSchenk2010}).\\[1ex]
{\it Illustration continued.} In the considered cases we derive with $s=-\sfs$
\begin{itemize}
 \item[\textit{(pp)}] 
If  $p>1/2$, $a>1/2$ and  $p+a\geq3/2$, then 
$\cR^{\Li_{t_0}^{(\sfs)}}_*[n^{-1};\cSowr,\cOpwd]\asymp n^{-(2p-2\sfs-1)/(2p+2a)}$.
 \item[\textit{(pe)}]  If $p>1/2$ and $a>0$, then 
$\cR^{\Li_{t_0}^{(\sfs)}}_*[n^{-1};\cSowr,\cOpwd]\asymp (\log n)^{-(2p-2\sfs-1)/(2a)}$.
 \item[\textit{(ep)}] If $p>0$ and $a>1/2$, then 
$\cR^{\Li_{t_0}^{(\sfs)}}_*[n^{-1};\cSowr,\cOpwd]\asymp n^{-1}(\log n)
^{(2a+2\sfs+1)/(2p)}$.
\end{itemize}
{\it Adaptive point-wise estimation of derivatives.}  
In the three considered cases the obtainable rate of the fully data-driven estimator $\hSo^{(\sfs)}_{\whm}(t_0)$  is given as follows:
\begin{itemize}
 \item[\textit{(pp)}] 
If  $p>1/2$, $a>1/2$ and  $p+a\geq3/2$, then\\ 
$
\cR^{\Li_{t_0}^{(\sfs)}}_*[(1+\log n)n^{-1};\cSowr,\cOpwd]
\asymp
 (n^{-1} \log n)^{(2p-2\sfs-1)/(2p+2a)}$.
 \item[\textit{(pe)}]  If $p>1/2$ and $a>0$, then\\ 
$\cR^{\Li_{t_0}^{(\sfs)}}_*[(1+\log n)n^{-1};\cSowr,\cOpwd]\asymp (\log n)^{-(2p-2\sfs-1)/2a}$.
 \item[\textit{(ep)}] If $p>0$ and $a>1/2$, then\\ 
$\cR^{\Li_{t_0}^{(\sfs)}}_*[(1+\log n)n^{-1};\cSowr,\cOpwd]\asymp n^{-1}(\log n)^{(2p+2a+2\sfs+1)/2p}$.
\end{itemize}
Also in the situation of adaptively estimating the $(q)$-th derivative in a given point the obtained rates  deteriorate  by a logarithmic factor  in the  cases \textit{(pp)} and \textit{(pe)} only.

\paragraph{Local average estimation.} Next we are  interested in the average value   of $\So $ on the interval $[0,b]$ for $b\in(0,1]$. If we denote the linear functional mapping $h$ to 
$b^{-1}\int_{0}^b h (t)dt$ by $\Li^b$, then it is easily seen that 
$[\Li^b]_1=1$, $[\Li^b]_{2j}=(\sqrt2\pi jb)^{-1}\sin(2\pi jb)$,
 $[\Li^b]_{2j+1}=(\sqrt2\pi j b)^{-1}\cos(2\pi jb)$  for $j\geq1$.  In this situation  the plug-in estimator $ \hLi^b_{m}= b^{-1}\int_0^b\hSo_m(t)dt$ is written as
\begin{equation*}
\hLi^{b}_m
=
\left\{\begin{array}{lcl} 
[\Li^b]_{\um}^t\fou{\hOp}_{\um}^{-1} [\widehat{g}]_{\um}, && \mbox{if $\fou{\hOp}_{\um}$ is non-singular and }\mnormV{\fou{\hOp}^{-1}_{\um}}\leq  n,\\
0,&&\mbox{otherwise}.
\end{array}\right.
\end{equation*}
{\it Minimax optimal estimation of local averages.} The estimator  $\hLi^b_{\moptn}$ attains the minimax optimal rate, i.e., $
\sup_{\So  \in \cF_\Sow^{\Sor}} \sup_{\Op \in \cOpwd}\Ex|\int_0^b\hSo_{\moptn}(t)dt -\int_0^b\So (t)dt|^2
\leq C
\cR^{\Li^{b}}_*[n^{-1};\cSowr,\cOpwd]
$ for $C>0$.\\[1ex]
{\it Illustration continued.} In the three cases  the order of  $\cR^{\Li^{b}}_*[n^{-1};\cSowr,\cOpwd]$ is given as follows:
\begin{itemize}
 \item[\textit{(pp)}] 
If  $p\geq 0$, $a>1/2$ and  $p+a>3/2$, then 
$
\cR^{\Li^{b}}_*[n^{-1};\cSowr,\cOpwd]\asymp n^{-(2p+1)/(2p+2a)}$.
 \item[\textit{(pe)}]  If $p\geq 0$ and $a>0$, then  
$\cR^{\Li^{b}}_*[n^{-1};\cSowr,\cOpwd]\asymp  (\log n)^{-(2p+1)/2a}$.
 \item[\textit{(ep)}] If $p>0$ and $a>1/2$, then 
$
\cR^{\Li^{b}}_*[n^{-1};\cSowr,\cOpwd]\asymp n^{-1}(\log n)^{(2a-1)/2p}$.
\end{itemize}
{\it Adaptive estimation of local averages.}
In the three considered cases the obtainable rate of the adaptive estimator $\hLi^b_{\whm}$  is given below:
\begin{itemize}
 \item[\textit{(pp)}] 
If  $p\geq 0$, $a>1/2$ and  $p+a>3/2$, then 
$
\cR^{\Li^{b}}_*[(1+\log n)n^{-1};\cSowr,\cOpwd] \asymp (n^{-1}\log n)^{(2p+1)/(2p+2a)}$.
 \item[\textit{(pe)}]  If $p\geq 0$ and $a>0$, then 
$
\cR^{\Li^{b}}_*[(1+\log n)n^{-1};\cSowr,\cOpwd]\asymp ~ (\log n)^{-(2p+1)/2a}$.
 \item[\textit{(ep)}] If $p>0$ and $a>1/2$, then 
$
\cR^{\Li^{b}}_*[(1+\log n)n^{-1};\cSowr,\cOpwd]\asymp n^{-1} (\log n)^{(2p+2a-1)/2p}$.
\end{itemize}
In this setting again, we notice a deterioration of logarithmic order  in the  cases \textit{(pp)} and \textit{(pe)} only.
\section*{Appendix}
\renewcommand{\thesubsection}{\Alph{subsection}}
\renewcommand{\theprop}{\Alph{subsection}.\arabic{prop}}
\numberwithin{equation}{subsection}  
\numberwithin{prop}{subsection}  
This section gathers preliminary technical results and the proofs of Proposition \ref{adaptive:p1} and \ref{adaptive:p2}.
\setcounter{subsection}{0}
\subsection{Notations}\label{app:notations} We begin by defining and recalling the notations which are used in the proofs. Given an integer $m\geq 1$, $\Sspace_m$ denotes the subspace of
$\Hspace$ spanned by the functions $\{\bas_1,\dotsc,\bas_m\}$. $\proj_m$ and $\proj_m^\perp$ denote the orthogonal projections on  $\Sspace_m$ and its orthogonal complement $\Sspace_m^\perp$
respectively. If $K$ is an operator mapping $\Hspace$ into itself and  we restrict $\proj_m K \proj_m$ to an  operator from $\Sspace_m$ into itself, then it can be represented by
the matrix $[K]_{\um}$.   Furthermore, $[\Diag_v]_{\um}$ and $\Id_{\um}$ denote the $m$-dimensional diagonal   matrix with diagonal entries $(v_j)_{1\leq j\leq m}$   and the identity
matrix respectively.  With a slight abuse of notations  $\normV{v}$ denotes the euclidean norm of the vector $v$. In particular, for all $f\in\Sspace_m$ we have $\normV{f}_v^2=
[f]_{\um}^t[\Diag_v]_\um[f]_{\um} =\normV{[\Diag_v]_\um^{1/2}[f]_{\um}}^2$. Moreover, we use the notations 
\begin{equation*}
\hV_m= \max\limits_{1\leq k\leq m}\fLi_\uk^t\fhOp_{\uk}^{-1}\fLi_\uk,\; V_m= \max\limits_{1\leq k\leq m}\fLi_\uk^t\fOp_{\uk}^{-1}\fLi_\uk,\; V_m^\Opw= \fLi_{\um}^t[\Diag_\Opw]_\um^{-1}\fLi_{\um}.
\end{equation*}
Recall that $[\hOp]_{\um}=\frac{1}{n}\sum_{i=1}^n[X_i]_{\um}[X_i]_{\um}^t$ and $[\hgf]_{\um}=\frac{1}{n}\sum_{i=1}^nY_i[X_i]_{\um}$ where $[\Op]_{\um}=\Ex \fou{X}_\um\fou{X}^t_\um$
and $[\gf]_\um=\Ex Y\fou{X}_\um$. Given a Galerkin solution $\So_m\in\Sspace_m$, let $U_m:=Y-
\HskalarV{\So_m,X}=\sigma\epsilon+\HskalarV{\So-\So_m,X}$. We introduce $\rho_m^2:=\Ex U^2_m=\sigma^2+\HskalarV{\Op(\So-\So_m),(\So-\So_m)}$, $\sigma_Y^2:=\Ex Y^2= \sigma^2+\HskalarV{\Op
  \So,\So}$ and $\sigma_m^2= 2\big(\sigma_Y^2+[\gf]_{\um}^t\fOp_{\um}^{-1}[\gf]_{\um}\big)$ where we use that $\epsilon$ and $X$ are uncorrelated. With these notations we
have
\begin{multline*}
\pen_m=100\sigma_m^2 V_m \lognn,\quad \hpen_m=700\hsigma_m^2\hV_m\lognn.\hfill
\end{multline*}
Let us define the random matrix $[\Xi]_{\um}$  and random vector $[W]_{\um}$, respectively, by 
\begin{equation*}
[\Xi]_{\um}:= [\Op]_{\um}^{-1/2}[\hOp]_{\um}[\Op]_{\um}^{-1/2} - \Id_{\um},
\quad\mbox{and}\quad [W]_{\um}:= [\hgf]_{\um}- [\hOp]_{\um} [\galsol]_{\um},
\end{equation*}
where $\Ex [\Xi]_{\um}= 0$, because $\Ex[\hOp]_{\um}=[\Op]_{\um}$, and $\Ex [W]_{\um} =[\Op(\So-\So_m)]_{\um}=0$. Furthermore, we introduce $\hsigma_Y^2:=n^{-1}\sum_{i=1}^nY_i^2$ and   the events
\begin{multline}\label{app:pre:e4}
 \hOpsetmn:=\{ \mnormV{[\hOp]^{-1}_{\um}}\leq n \},\quad  \Xisetmn:= \{8\sqrt{m}\mnormV{[\Xi]_{\um}}\leq 1\},\\
\asetn:=\{{1}/{2}\leq \hsigma_Y^2/\sigma_Y^2\leq {3}/{2}\},\quad 
\bsetn:=\{\mnormV{\fou{\Xi}_\uk}\leq 1/8,\forall 1\leq k\leq \Mln\},\\
\csetn:=\{[W]_\uk^t\fOp_{\uk}^{-1}[W]_\uk\leq
  \frac{1}{8}(\fou{\gf}_\uk^t\fOp^{-1}_\uk\fou{\gf}_\uk + \sigma_Y^2) ,\forall 1\leq k\leq \Mln\},\hfill
\end{multline}
along with their respective complements $\hOpsetmn^c$, $\Xisetmn^c$, $\asetn^c$, $\bsetn^c$, and $\csetn^c$. Here and subsequently,   we will denote by $C$ a universal numerical constant  and
by $C(\cdot)$ a constant depending only on the arguments. In both cases, the values of the constants may change with every appearance.
\subsection{Preliminary results}\label{app:pre}
The proof of
the next lemma can be found in \cite{JohannesSchenk2010}. It relies on  the properties of the sequences $\Sow$, $\Opw$ and $[\Li]$ given in  
Assumption \ref{minimax:ass:reg}.
\begin{lem}\label{minimax:l1}
 Let $\Opother$ belong to $\cOpwd$ where the sequence $\Opw$ satisfies Assumption \ref{minimax:ass:reg}, then we have 
\begin{gather}\label{minimax:l1:e1}
\sup_{m\in\Nz}\Bigl\{ \Opw_{m} \normV{[\Opother]_{\um}^{-1}}\Bigr\}\leq 4\Opd^3,\\
\label{minimax:l1:e2}
\sup_{m\in\Nz}\normV{[\Diag_{\Opw}]^{1/2}_{\um}[\Opother]_{\um}^{-1}[\Diag_{\Opw}]^{1/2}_{\um}}\leq 4\Opd^3,\\ \label{minimax:l1:e3}
\sup_{m\in\Nz}\normV{[\Diag_\Opw]^{-1/2}_{\um}[\Opother]_{\um}[\Diag_\Opw]^{-1/2}_{\um}}\leq \Opd.
\end{gather}
Consider in addition  $\So\in \cSowr$  with  sequence $\Sow$  satisfying Assumption \ref{minimax:ass:reg}. If  $\So_m$  denotes a Galerkin solution of $\gf=\Opother\So$ then for any strictly positive sequence $w:=(w_j)_{j\geq1}$ such
 that $w/\Sow$ is  non-increasing we obtain for all $m\in\Nz$
\begin{gather}  \label{minimax:l1:e4}
\normV{\So-\So_{m}}_w^2 \leq   34\, \Opd^8 \, \Sor\,  \frac{w_{m}}{\Sow_{m}}\max\bigg( 1, \frac{\Opw_{m}^2}{w_{m}} \max_{1\leq j\leq m}\bigg\{\frac{w_j}{\Opw_j^2}\bigg\} \bigg),\\\label{minimax:l1:e5}
\normV{\So_m}^2_\Sow\leq  34\, \Opd^8 \, \Sor,\quad \HnormV{\Opother^{1/2}(\So-\So_m)}^2\leq  34\, \Opd^9 \, \Sor\,
\Opw_m\Sow_m^{-1}.
\end{gather}
Furthermore, under Assumption \ref{minimax:ass:reg}   we have
\begin{gather} \label{minimax:l1:e7}
|{\Li(\So-\So_m)}|^2\leqslant
2\,\Sor\,\bigg\{\sum_{j>m}\frac{\fLi_j^2}{\Sow_j}+ 2(1+\Opd^4)\frac{\Opw_{m}}{\Sow_{m}}\sum_{j=1}^{m}\frac{\fLi_j^2}{\Opw_j}\bigg\}.
\end{gather}
\end{lem}
\begin{lem}\label{app:pre:l1}Let Assumption \ref{minimax:ass:reg} be satisfied and define $D:=(4\Opd^3)$. For $\Op\in\cOpwd$ we have
\begin{itemize}
\item[(i)] $\Opd^{-1} \leq V_m/V^\Opw_m\leq D $, $\Opd^{-1} \leq \Opw_m \mnormV{\fOp_\um^{-1}}\leq  D$ and $\Opd^{-1} \leq \Opw_m \max_{1\leq k\leq m}\mnormV{\fOp_\uk^{-1}}\leq  D$ for all $m\geq1$,
\item[(ii)] $V_{\Mon}^\Opw \leq n 4D(1+\log n)^{-1}$ and hence $V_{\Mon} \leq n 4D^2(1+\log n)^{-1}$  for all $n\geq 1$,
\item[(iii)]  $ 2\max_{1\leq m\leq \Mon}\normV{\fou{\Op}^{-1}_{\um}}\leq  n$ if  $n\geq 2D$ and  $\normV{\fLi_{\underline{\Mon}}}^2(1+\log n)\geq 8D^2$.
\end{itemize}
If   $\So$ belongs in addition to $\cSowr$ then it holds   for all $m\geq1$
\begin{itemize}
\item[(iv)]  $\rho_m^2\leq \sigma_m^2\leq 2(\sigma^2+35 \Opd^9\Sor)$ and
\item[(v)] $\sup_{\So\in\cSowr}\sup_{\Op\in\cOpwd}\set{\pen_m+\bias_m}\leq 202D^4\,(\sigma^2+\Sor)\,\cR^\Li_m(\lognn;\cSowr,\cOpwd)$.
\end{itemize}
\end{lem}
\begin{proof}[\noindent\textcolor{darkred}{\sc Proof of Lemma \ref{app:pre:l1}.}]  Due to \eqref{minimax:l1:e2} - \eqref{minimax:l1:e3} in Lemma
  \ref{minimax:l1}, we have $V_m\leq 4\Opd^3\fLi_\um^t[\Diag_{\Opw}]^{-1}_{\um}\fLi_\um$ $= D V_m^\Opw$ and $V_m^\Opw\leq \Opd
  \fLi_\um^t\fOp^{-1}_{\um}\fLi_\um\leq\Opd V_m$. Moreover, from \eqref{minimax:l1:e1} and \eqref{minimax:l1:e2} it follows that $\mnormV{\fOp_\um^{-1}}\leq 4\Opd^3\Opw_m^{-1}$ and
  $\Opw_m^{-1}\leq \Opd \mnormV{\fOp_\um^{-1}}$.  Thus,  for all $m\geq 1$ we have $D\geq \mnormV{\fOp_\um^{-1}}\Opw_m \geq   \Opd^{-1}$.
Hence, the monotonicity of $\Opw$ implies $\Opd^{-1} \leq \Opw_M \max_{1\leq m\leq M}\mnormV{\fOp_\um^{-1}}\leq D$. From these estimates we obtain (i).
\\
\noindent{Proof of  (ii).}  Observe that $V_{\Mon}^\Opw\leq \normV{\fLi_{\underline{\Mon}}}^2\Opw_{\Mon}^{-1}$. In case $\Mon=1$ the assertion  is trivial, since $\fLi_1^2=\Opw_1$  due to
Assumption \ref{minimax:ass:reg}. Thus, consider   $\Mln\geq \Mon>1$, which implies $\min_{1\leq j\leq \Mon}\{ \Opw_j\normV{\fLi_{\underline{\Mon}}}^{-2}\}\geq(1+\log n)/(4Dn)$, and hence
$V_{\Mon}^\Opw\leq  4Dn(1+\log n)^{-1}$. Moreover, from (i) follows  $V_{\Mon}\leq D V_{\Mon}^\Opw\leq  4D^2n(1+\log n)^{-1}$, which proves (ii).\\ 
\noindent{Proof of (iii).} By employing that  $ D\Opw_{\Mon}^{-1} \geq  \max_{1\leq m\leq \Mon}\normV{\fou{\Op}_\um^{-1}}$,  the assertion  (iii)  follows
in case $\Mon=1$  from $\Opw_1=1$, while in case $\Mon>1$,  we use  $\normV{\fLi_{\underline{\Mon}}}^2/\Opw_{\Mon}\leq 4Dn/(1+\log n)$.\\
\noindent{Proof of (iv).} Since $\epsilon$ and $X$ are
centered it follows from $[\So_m]_\um=\fOp_{\um}^{-1}[\gf]_{\um}$ that $\rho_{m}^2\leq 2\big(\Ex Y^2+  \Ex|\HskalarV{\So_m,X}|^2\big) = 2\big(\sigma_Y^2+[\gf]_{\um}^t[\Op]_{\um}^{-1}[\gf]_{\um}\big) =\sigma_m^2$.
Moreover,  by employing successively the inequality of \cite{Heinz1951}, i.e. $\normV{\Op^{1/2}\So}^2\leq \Opd\normV{\So}_\Opw^2$, and Assumption \ref{minimax:ass:reg}, i.e.,  $\Opw$ and  $\Sow^{-1}$ are non-increasing, the identity $\sigma_Y^2=\sigma^2+\HskalarV{\Op \So,\So}$ implies  
\begin{equation}
\sigma_Y^2\leq \sigma^2+\Opd\normV{\So}_\Opw^2\leq
\sigma^2+\Opd\Sor.\label{app:pre:l1:e1}
\end{equation}
Furthermore, \eqref{minimax:l1:e3} and \eqref{minimax:l1:e4}
 in Lemma \ref{minimax:l1} imply
 \begin{equation}\label{app:pre:l1:e2}
[\gf]_{\uk}^t\fOp_{\uk}^{-1}[\gf]_{\uk}\leq \Opd
 \normV{\So_k}_\Opw^2 \leq 34\Opd^9\Sor.
\end{equation}
The assertion (iv) follows now by  combination of  the estimates \eqref{app:pre:l1:e1} and \eqref{app:pre:l1:e2}. \\
\noindent{Proof of (v).} From $V_{m}\leq D V_m^\Opw$ due to assertion (i) and the second inequality in (iv) we derive
\begin{equation}\label{app:pre:l1:e3}
 \pen_{m}\leq 100 \sigma_m^2 \lognn D V_m^\Opw\leq 200 (\sigma^2 + \Sor) D^4 \lognn \sum_{j=1}^{m}\fLi_j^2\Opw_j^{-1}.
\end{equation}
Furthermore, by using \eqref{minimax:l1:e7} in Lemma \ref{minimax:l1} we obtain   that
\begin{equation}\label{app:pre:l1:e4}
\bias_m\leq 16\Opd^4\,\Sor\,\{  \max(\sum_{j>m}\fLi_j^2\Sow_j^{-1},{\Opw_{m}}{\Sow_m^{-1}}\sum_{j=1}^{m}\fLi_j^2\Opw_j^{-1})\}.
\end{equation} Combining  the bounds \eqref{app:pre:l1:e3} and \eqref{app:pre:l1:e4}  implies  assertion (v), which completes the proof.\end{proof}

\begin{lem}\label{app:pre:l2}For all $n,m\geq 1$ we have 
\[\set{\frac{1}{4}< \frac{\mnormV{\fhOp_\um^{-1}}}{\mnormV{\fOp_\um^{-1}}}\leq 4,\forall\,1\leq m\leq \Mln}\subset \set{\Mun\leq \Men\leq \Mon}.\]
\end{lem}
\begin{proof}[\noindent\textcolor{darkred}{\sc Proof  of Lemma \ref{app:pre:l2}.}] Let $\htau_m=\mnormV{\fhOp_\um^{-1}}^{-1}$ and recall that $1\leq \Men\leq \Mln$ with
\begin{equation*}
\set{\Men = M}=\left\{\begin{matrix}
&&\set{\frac{\htau_{M+1}}{\normV{\fLi^2_{\underline{M+1}}}^2}< \frac{1+\log n}{n}},& M=1,\\
\set{\min\limits_{2\leq m\leq M}\frac{\htau_{m}}{\normV{\fLi^2_{\um}}^2}\geq  \frac{1+\log
    n}{n}}&\bigcap&\set{\frac{\htau_{M+1}}{\normV{\fLi^2_{\underline{M+1}}}^2} < \frac{1+\log n}{n}},&1< M< \Mln,\\
\set{\min\limits_{2\leq m\leq M} \frac{\htau_{m}}{\normV{\fLi^2_{\underline{m}}}^2}\geq  \frac{1+\log
    n}{n}},&&&M=\Mln.
\end{matrix}\right.
\end{equation*}
Given  $\tau_m^{-1}:=\mnormV{\fOp_\um^{-1}}$ we have $D^{-1}\leq \tau_m/\Opw_m \leq \Opd$ for all $m\geq 1$ due to (i)  in Lemma
\ref{app:pre:l1} which we use to proof the following two assertions 
\begin{align}\label{app:pre:l2:e1}
\set{\Men < \Mun}&\subset \set{\min\limits_{1\leq m\leq \Mln}:\frac{\htau_{m}}{\tau_{m}}< \frac{1}{4}},\\\label{app:pre:l2:e2}
\set{\Men > \Mon}&\subset \set{\max_{1\leq m\leq \Mln}\frac{\htau_{m}}{\tau_{m}}\geq 4}.
\end{align}
 Obviously, the assertion of the Lemma follows now  by combination of \eqref{app:pre:l2:e1} and \eqref{app:pre:l2:e2}.\\
Consider \eqref{app:pre:l2:e1} which is trivial in case $\Mun=1$. For $\Mun>1$ we have $\min\limits_{1\leq m\leq
  \Mun} \frac{\Opw_m}{\normV{\fLi^2_{\um}}^2}\geq \frac{4D(1+\log n)}{n}$ and, hence  $\min\limits_{1\leq m\leq \Mun} \frac{\tau_m}{\normV{\fLi^2_{\um}}^2}\geq \frac{4(1+\log
  n)}{n}$. By exploiting the last estimate  we obtain 
\begin{multline*}
\set{\Men < \Mln}\cap\set{\Men < \Mun}=\bigcup_{M=1}^{\Mun-1}\set{\Men =M}
\\\hfill\subset\bigcup_{M=1}^{\Mun-1}\set{\frac{\htau_{M+1}}{\normV{\fLi^2_{\underline{M+1}}}^2}<
   \frac{1+\log n}{n}}
=
\set{\min_{2\leq m\leq \Mun}\frac{\htau_{m}}{\normV{\fLi^2_{\um}}^2}< \frac{1+\log n}{n}}\\\subset  \set{\min_{1\leq m\leq
  \Mun}\frac{\htau_{m}}{\tau_{m}}< 1/4}
\end{multline*}
while  trivially $\set{\Men = \Mln}\cap\set{\Men < \Mun}=\emptyset$ which proves \eqref{app:pre:l2:e1} because $\Mun\leq \Mln$. 
Consider \eqref{app:pre:l2:e2} which  is trivial  in case $\Mon=\Mln$. If $\Mon<\Mln$, then
$\frac{\tau_{\Mon+1}}{\normV{\fLi^2_{\underline{\Mon+1}}}^2}< \frac{(1+\log
  n)}{4n}$, and hence 
\begin{multline*}
  \set{\Men >1}\cap\set{\Men> \Mon}=\bigcup_{M=\Mon+1}^{\Mln}\set{\Men =M}\\
\subset\bigcup_{M=\Mon+1}^{\Mln}\set{\min_{2\leq m\leq M}
    \frac{\htau_{m}}{\normV{\fLi^2_{\um}}^2}\geq \frac{1+\log n}{n}}=
      \set{\min_{2\leq m\leq (\Mon+1)}\frac{\htau_{m}}{\normV{\fLi^2_{\um}}^2}\geq \frac{1+\log n}{n}}\\
\subset\set{\frac{\htau_{\Mon+1}}{\tau_{\Mon+1}}\geq 4}
\end{multline*}
while  $\set{\Men = 1}\cap\set{\Men > \Mon}=\emptyset$ which shows \eqref{app:pre:l2:e2} and completes the proof.\end{proof}
\begin{lem}\label{app:pre:l3}  
Let $\asetn$, $\bsetn$ and $\csetn$ as in \eqref{app:pre:e4}.   For all $n\geq1$ it  holds true that \\ \centerline{$\asetn\cap\bsetn\cap\csetn\subset\{ \pen_k\leq\hpen_k\leq 24\pen_k,1\leq k\leq \Mln\}\cap\{\Mun\leq \Men \leq  \Mon\}.$} 
\end{lem}
\begin{proof}[\noindent\textcolor{darkred}{\sc Proof of Lemma \ref{app:pre:l3}.}] 
Let $\Mln\geq k\geq1$. If $\mnormV{\fou{\Xi}_\uk}\leq 1/8$, i.e., on the event $\bsetn$, it is easily verified that $\mnormV{(\Id_\uk +\fou{\Xi}_\uk)^{-1}-\Id_\uk}\leq 1/7$
which we exploit to conclude
\begin{multline}\label{app:pre:l3:e1}
(6/7) \mnormV{\fOp^{-1}_\uk}\leq  \mnormV{\fhOp^{-1}_\uk}\leq (8/7) \mnormV{\fOp^{-1}_\uk} \quad\mbox{and}\\
 (6/7) s^t\fOp^{-1}_\uk s \leq s^t\fhOp^{-1}_\uk s \leq (8/7) s^t\fOp^{-1}_\uk s,\quad\mbox{for all } s\in\Rz^k,
\end{multline}
and, consequently 
\begin{equation}\label{app:pre:l3:e1:2}
 (6/7) [\hgf]^t_\uk[\Op]^{-1}_\uk[\hgf]_\uk \leq [\hgf]^t_\uk[\hOp]^{-1}_\uk[\hgf]_\uk\leq (8/7) [\hgf]^t_\uk[\Op]^{-1}_\uk[\hgf]_\uk.
\end{equation}
Moreover, from $\mnormV{\fou{\Xi}_\uk}\leq 1/8$ we obtain after some algebra,
\begin{multline*}
[\gf]^t_\uk[\Op]^{-1}_\uk[\gf]_\uk \leq \frac{1}{16} [\gf]^t_\uk[\Op]^{-1}_\uk[\gf]_\uk +  4[W]_\uk\fou{\Op}_\uk^{-1}[W]_{\uk} + 2 [\hgf]^t_\uk[\Op]^{-1}_\uk[\hgf]_\uk,\\
[\hgf]^t_\uk[\Op]^{-1}_\uk[\hgf]_\uk \leq \frac{33}{16} [\gf]^t_\uk[\Op]^{-1}_\uk[\gf]_\uk +  4[W]_\uk\fou{\Op}_\uk^{-1}[W]_{\uk}.\hfill
\end{multline*}
Combining each of these estimates with \eqref{app:pre:l3:e1:2} yields
\begin{multline*}
(15/16) [\gf]^t_\uk[\Op]^{-1}_\uk[\gf]_\uk\leq 4[W]_\uk\fou{\Op}_\uk^{-1}[W]_{\uk} + (7/3)[\hgf]^t_\uk[\hOp]^{-1}_\uk[\hgf]_\uk,\\
(7/8)[\hgf]^t_\uk[\hOp]^{-1}_\uk[\hgf]_\uk\leq (33/16)  [\gf]^t_\uk[\Op]^{-1}_\uk[\gf]_\uk + 4[W]_\uk\fou{\Op}_\uk^{-1}[W]_{\uk}.\hfill
\end{multline*}
If in addition $[W]_\uk^t\fou{\Op}_{\uk}^{-1}[W]_\uk\leq \frac{1}{8}(\fou{\gf}_\uk^t\fou{\Op}^{-1}_\uk\fou{\gf}_\uk + \sigma_Y^2) $, i.e., on the event $\csetn$, then the last two
estimates imply respectively
\begin{multline*}
(7/16) ([\gf]^t_\uk[\Op]^{-1}_\uk[\gf]_\uk+\sigma_Y^2) \leq (15/16)\sigma_Y^2 + (7/3)
[\hgf]^t_\uk[\hOp]^{-1}_\uk[\hgf]_\uk,\\
(7/8)[\hgf]^t_\uk[\hOp]^{-1}_\uk[\hgf]_\uk\leq (41/16)  [\gf]^t_\uk[\Op]^{-1}_\uk[\gf]_\uk +  (1/2)\sigma_Y^2,\hfill
\end{multline*}
and hence in case ${1}/{2}\leq \hsigma_Y^2/\sigma_Y^2\leq {3}/{2}$, i.e., on the event $\asetn$, we obtain
\begin{multline*}
(7/16)( [\gf]^t_\uk[\Op]^{-1}_\uk[\gf]_\uk +\sigma_Y^2)\leq  (15/8) \hsigma_Y^2 + (7/3)
[\hgf]^t_\uk[\hOp]^{-1}_\uk[\hgf]_\uk,\\
(7/8)([\hgf]^t_\uk[\hOp]^{-1}_\uk[\hgf]_\uk + \hsigma^2_Y)\leq (41/16)  [\gf]^t_\uk[\Op]^{-1}_\uk[\gf]_\uk + (29/16)\sigma_Y^2.\hfill
\end{multline*}
Combining the last two estimates yields 
\[\frac{1}{6} (2[\gf]^t_\uk[\Op]^{-1}_\uk[\gf]_\uk +2\sigma_Y^2)\leq (2[\hgf]^t_\uk[\hOp]^{-1}_\uk[\hgf]_\uk + 2\hsigma^2_Y)\leq 3 (2[\gf]^t_\uk[\Op]^{-1}_\uk[\gf]_\uk +2\sigma_Y^2).\]
Since 
the last estimate and \eqref{app:pre:l3:e1} hold for all $1\leq k\leq \Mln$
on the event $\asetn\cap\bsetn\cap\csetn$  it follows
\[\asetn\cap\bsetn\cap\csetn\subset\set{  \frac{1}{6} \sigma_m^2\leq \hsigma_m^2\leq  3 \sigma_m^2  \mbox{ and } (6/7)V_m \leq \hV_m\leq (8/7) V_m, \,\forall1\leq m\leq
  \Mln}.\]
The definitions of $\pen_m=100\sigma_m^2V_m\lognn$ and $\hpen_m=700\hsigma_m^2\hV_m\lognn$ imply
\begin{equation}\label{app:pre:l3:e2}
\asetn\cap\bsetn\cap\csetn\subset\set{\pen_m \leq \hpen_m\leq 24\pen_m, \,\forall1\leq m\leq\Mln}.
\end{equation}
On the other hand, by exploiting successively \eqref{app:pre:l3:e1} and Lemma \ref{app:pre:l2}  we obtain
\begin{equation}\label{app:pre:l3:e3}
\asetn\cap\bsetn\cap\csetn\subset\set{\frac{6}{7}\leq \frac{\mnormV{\fhOp^{-1}_\um}}{\mnormV{\fOp^{-1}_\um}}\leq \frac{8}{7}, \,\forall1\leq m\leq \Mln}\subset \set{\Mun\leq
  \Men\leq \Mon}.
\end{equation}
From \eqref{app:pre:l3:e2} and \eqref{app:pre:l3:e3} follows  the assertion of the lemma, which completes the proof.
\end{proof}
\begin{lem}\label{app:pre:l4}  
For all $m,n\geq1$ with $ n\geqslant (8/7) \mnormV{\fOp^{-1}_{\um}}$  we have $\Xisetmn \subset\hOpsetmn$.
\end{lem}
\begin{proof}[\noindent\textcolor{darkred}{\sc Proof of Lemma \ref{app:pre:l4}.}] 
Taking  the identity $[\hOp]_{\um}= [\Op]^{1/2}_{\um}\{\Id_{\um}+[\Xi]_{\um}\}[\Op]^{1/2}_{\um}$ into account, we observe that 
 $\sqrt{m}\mnormV{[\Xi]_{\um}}\leqslant 1/8$    implies $\mnormV{\fhOp^{-1}_{\um}} \leqslant \frac{8\sqrt{m}}{8\sqrt{m}-1}\mnormV{\fOp^{-1}_{\um}} \leqslant(8/7)\mnormV{\fOp^{-1}_{\um}}$ due to the usual Neumann series argument.
If $ n\geqslant (8/7) \mnormV{[\Op]^{-1}_{\um}}$, then the last assertion implies  $\Xisetmn \subset\hOpsetmn$, which proves the lemma.
\end{proof}
\subsection{Preliminary results due to the normality assumption}\label{app:gauss}
We will suppose throughout this section that the conditions of Theorem  \ref{adaptive:t1} 
and in particular Assumption \ref{minimax:ass:reg}
are satisfied, thus, the technical Lemmas  stated in  Section \ref{app:pre} are applicable. We show  technical assertions under the  assumption of normality (Lemmas \ref{app:gauss:l1}-
\ref{app:gauss:l4}) which  are used below to prove  Propositions \ref{adaptive:p1}
and \ref{adaptive:p2}.  

We begin by  recalling  elementary properties due to the assumption that $X$
and $\epsilon$ are jointly normally distributed, which are frequently used  in the following proofs. For any $h\in\Hz$ the random
variable $\HskalarV{h,X}$  is normally distributed with mean zero and variance $\HskalarV{\Op h,h}$. Consider  the Galerkin solution $\So_m$ and  $h\in\Sspace_m$ then  the
random variables $\HskalarV{\So-\So_m,X}$ and $\HskalarV{h,X}$  are independent. Thereby,  $U_m=Y-\HskalarV{\So_m,X}=\sigma\epsilon +
  \HskalarV{\So-\So_m,X}$  and $\fou{X}_\um$ are independent, normally distributed with mean zero, and, respectively,  variance $
  \rho^2_m$ and covariance matrix $\fOp_\um$. Consequently,
  $(\rho_m^{-1}U_m,\fou{X}_\um^t\fOp_\um^{-1/2})$ is a $(m+1)$-dimensional vector of  i.i.d.\ standard normally distributed random variables. 
Let us further state elementary inequalities for Gaussian random variables.
\begin{lem}\label{app:gauss:l1}Let $\{\lU_i,\lV_{ij},1\leq i\leq n,1\leq j\leq m\}$  be independent and standard normally distributed random variables. We have for all $\eta>0$  and
  $\zeta\geq 4m/n$
  \begin{align}\label{app:gauss:l1:e1}
&P\vect{|n^{-1/2}\sum_{i=1}^n (\lU_i^2-1)|\geq \eta }\leq 2 \exp\bigg(-\frac{\eta^2}{8(1+\eta\,n^{-1/2})}\bigg);\\\label{app:gauss:l1:e2}
&P\vect{|n^{-1}\sum_{i=1}^n \lU_i\lV_{i1}|\geq \eta }\leq \frac{\eta n^{1/2}+2}{\eta n^{1/2} } \exp\bigg(-\frac{n}{4}\min\set{\eta^2,\frac{1}{4}}\bigg);\\\label{app:gauss:l1:e2:1}
&P\vect{n^{-2}\sum_{j=1}^m\absV{\sum_{i=1}^n \lU_i\lV_{ij}}^2\geq \zeta }\leq \exp\Bigl(-\frac{n}{16}\Bigr) +\exp\Bigl(-\frac{\zeta n}{64}\Bigr) ;
\end{align}
and for all $c>0$  and  $a_1,\dotsc,a_m\geq 0$ that
\begin{align}\label{app:gauss:l1:e3}
&\Ex\vectp{ n^{-1}\sum_{i=1}^n \lU_i^2-2}\leq  \frac{16}{n}\exp\bigg(-\frac{n}{16}\bigg) ; \\\label{app:gauss:l1:e4}
& \Ex\vectp{ |n^{-1/2}\sum_{i=1}^n \lU_i\lV_{i1}|^2- 4 c (1+\log n)}\leq  \frac{2n^{-c}}{e^c\sqrt{\pi c (1+\log n)}}  + 32c \exp\bigg(-\frac{n}{16}\bigg);\\\label{app:gauss:l1:e5}
& \Ex\vect{ \sum_{j=1}^m a_j\absV{\sum_{i=1}^n U_iV_{ij}}^2}^4\leq n^4\Bigl(11\sum_{j=1}^ma_j\Bigr)^4.
  \end{align}
\end{lem}
\begin{proof}[\noindent\textcolor{darkred}{\sc Proof of Lemma \ref{app:gauss:l1}.}]Define  $\lW:=\sum_{i=1}^n \lU_i^2$ and  $ \lZ_j:=(\sum_{i=1}^n\lU_i^2)^{-1/2}\sum_{i=1}^n
  \lU_i\lV_{ij}$. Obviously, $\lW$ has a $\chi^2$ distribution with $n$ degrees of freedom and $\lZ_1,\dotsc,\lZ_m$ given $\lU_1,\dotsc,\lU_n$ are independent and  standard normally distributed, which we use below without further reference.
The estimate (\ref{app:gauss:l1:e1}) is given in \cite{DahlhausPolonik2006} (Proposition A.1) and by using (\ref{app:gauss:l1:e1}) we have
  \begin{multline*}
P(|n^{-1}\sum_{i=1}^n \lU_i\lV_{i1}|\geq \eta )\leq P(n^{-1} \lW \geq 2) + \Ex\big[P\big(2n^{-1}|\lZ_1|^2\geq \eta^2 \big| \lU_1,\dotsc,\lU_n\big)\big]\\\hfill\leq
\exp\bigg(-\frac{n}{16}\bigg) + \frac{2}{\sqrt{\pi \eta^2 n }}\exp\bigg(-\frac{\eta^2 n}{4}\bigg),
\end{multline*}
which implies (\ref{app:gauss:l1:e2}). The estimate (\ref{app:gauss:l1:e2:1}) follows analogously and we omit the details. By using
\eqref{app:gauss:l1:e1} we obtain \eqref{app:gauss:l1:e3} as follows
\begin{multline*}
  \Ex\vectp{ n^{-1}\sum_{i=1}^n \lU_i^2-2} = \int_0^\infty P(n^{-1/2}\sum_{i=1}^n (\lU_i^2-1)\geq n^{1/2}(1+t))dt\\\hfill\leq \int_0^\infty
  \exp\bigg(-\frac{n(1+t)^2}{8(1+(1+t))}\bigg)dt \leq \int_0^\infty
  \exp\bigg(-\frac{n(1+t)}{16}\bigg)dt\\\hfill
=\exp\bigg(-\frac{n}{16}\bigg)   \int_0^\infty
  \exp\bigg(-\frac{n}{16}t\bigg)dt= \frac{16}{n}\exp\bigg(-\frac{n}{16}\bigg).  
\end{multline*}
Consider \eqref{app:gauss:l1:e4}. Since $n^{-1/2}\sum_{i=1}^n \lU_i$ is standard normally distributed, we have
\begin{multline*}
  \Ex\vectp{ |n^{-1/2}\sum_{i=1}^n \lU_i|^2- 2 c (1+\log n)} = \int_0^\infty P(|n^{-1/2}\sum_{i=1}^n \lU_i
|\geq (t+  2 c (1+\log n))^{1/2})dt\\\hfill\leq \int_0^\infty\frac{2}{\sqrt{2\pi (t+  2 c(1+\log n))}}
  \exp\bigg(-\frac{(t+  2 c(1+ \log n))}{2}\bigg)dt \\\hfill
\leq\frac{e^{-c} n^{-c}}{\sqrt{\pi c(1+ \log n)}} \int_0^\infty
  \exp\bigg(-\frac{1}{2}t\bigg)dt= \frac{2e^{-c}n^{-c}}{\sqrt{\pi c(1+ \log n)}}. 
\end{multline*}
By using the last bound and  (\ref{app:gauss:l1:e3}) we get 
\begin{multline*}
  \Ex\vectp{ |n^{-1/2}\sum_{i=1}^n \lU_i\lV_{i1}|^2- 4 c (1+\log n)} \\\hfill\leq   \Ex\bigg[n^{-1}\lW \Ex\big[\vect{|\lZ_1|^2- 2 c (1+\log n)}_+|\lU_1,\dotsc,\lU_n\big] + 2c(1+\log n) \vect{n^{-1}\lW -2}_+\bigg]\\
\leq  \frac{2n^{-c}}{e^c\sqrt{\pi c (1+\log n)}}  + 32c  \frac{(1+\log n)}{n}\exp\bigg(-\frac{n}{16}\bigg) 
\end{multline*}
which shows \eqref{app:gauss:l1:e4}. Finally, by applying $\Ex[Z_j^8|\lU_1,\dotsc,\lU_n]=105$ and $\Ex W^4=n(n+2)(n+4)(n+6)$ we obtain $\Ex[W^4Z_j^8]\leq
(11n)^4$ and hence 
  \begin{multline*}
\Ex\vect{ \sum_{j=1}^m a_j\absV{\sum_{i=1}^n U_iV_{ij}}^2}^4=\Ex \bigg( \sum_{j=1}^ma_j W Z_j^2\bigg)^4
\leq \bigg|\sum_{j=1}^m a_j (\Ex [ W^4Z_{j}^8])^{1/4} \bigg|^4 \leq (11n)^4  (\sum_{j=1}^ma_j)^4
  \end{multline*}
which shows \eqref{app:gauss:l1:e5} and completes the proof.
\end{proof}
\begin{lem}\label{app:gauss:l2} For all $n,m\geq 1$  we have
\begin{align}\label{app:gauss:l2:e1}
&n^{4}m^{-4}\Ex\mnormV{\fou{\Xi}_\um\fou{\Op}_\um^{1/2} }^8\leq (34\Ex\HnormV{X}^2)^4;\\\label{app:gauss:l2:e2}
&n^4\rho_m^{-8}\Ex\normV{\fou{W}_\um}^8\leq  (11\Ex\HnormV{X}^2)^4.
\intertext{Furthermore, there exists a numerical constant $C$ such that for all $n\geq 1$}\label{app:gauss:l2:e3}
&n^8\max_{1\leq m\leq\gauss{n^{1/4}}} P\vect{\frac{(\fou{W}_m^t\fou{\Op}_\um^{-1}\fou{W}_\um)}{\rho_m^2}>\frac{1}{16}}\leq  C; \\\label{app:gauss:l2:e4}
& n^8  \max_{1\leq m\leq \gauss{n^{1/4}}} P\vect{\sqrt{m}\mnormV{\fou{\Xi}_\um}> \frac{1}{8}}\leq C;\\\label{app:gauss:l2:e5}
&n^7P\vect{\{1/2\leq\hsigma^2_Y/\sigma^2_Y\leq3/2\}^c}\leq C;\\
\label{app:gauss:l2:e6}
&n^2\sup_{m\geq1}\Ex\vectp{ \frac{n(\fou{W}_m^t\fOp_m^{-1}\fou{W}_m)}{m\rho_m^2} - 8 (1+\log n) }\leq  C;\\\label{app:gauss:l2:e7}
&n^2\sup_{m\geq1}\Ex\vectp{ \frac{n(\fLi_m^t\fOp_m^{-1}\fou{W}_m)^2}{\rho_m^2\fLi_m^t\fOp_m^{-1}\fLi_m} - 8 (1+\log n) }\leq C.
\end{align}
\end{lem}
\begin{proof}[\noindent\textcolor{darkred}{\sc Proof of Lemma \ref{app:gauss:l2}.}]
Let $n,m\geq 1$ be fixed and  denote by $(\lambda_j,e_j)_{1\leq j\leq m}$ an eigenvalue decomposition of  $\fOp_\um$. Define   $\lU_i:=(\sigma\epsilon_i +
\HskalarV{\So-\So_m,X_i})/\rho_m$ and $\lV_{ij}:=(\lambda_j^{-1/2}e_j^t[X_i]_{\um})$, $1\leq i\leq n$, $1\leq j\leq m$, where $\lU_1,\dotsc,\lU_n,\lV_{11},\dotsc,\lV_{nm}$ are  independent and
standard normally distributed random variables.\\

Proof of \eqref{app:gauss:l2:e1}. For all $1\leq j,l\leq m$ let $\delta_{jl}=1$ if $j=l$ and zero otherwise. It is easily verified that
$\mnormV{\fou{\Xi}_\um\fou{\Op}_\um^{1/2}}^2\leq \sum_{j=1}^m\sum_{l=1}^m\lambda_l|n^{-1}\sum_{i=1}^n(\lV_{ij}\lV_{il}-\delta_{jl})|^2$. Moreover, for $j\ne l$ we have
$\Ex|\sum_{i=1}^n\lV_{ij}\lV_{il}|^8\leq (11n)^4$ by employing \eqref{app:gauss:l1:e5}  in Lemma \ref{app:gauss:l1} (take  $m=1$ and $a_1=1$), while $\Ex|\sum_{i=1}^n(\lV_{ij}^2-1)|^8=n^4 256(105/16+595/(2n)+ 1827/n^2+2520/n^3)\leq (34n)^4$.
From these estimates  we get by successively  employing Jensen's
and Minkowski's inequality that  \[m^{-4}\Ex\mnormV{\fou{\Xi}_\um\fou{\Op}_\um^{1/2} }^8\leq
n^{-8}m^{-1}\sum_{j=1}^m\big(\sum_{l=1}^m\lambda_l(\Ex|\sum_{i=1}^n(\lV_{ij}\lV_{il}-\delta_{jl})|^8)^{1/4}\big)^4\leq n^{-4}(34\sum_{j=1}^m\lambda_j)^4.\]
The last estimate together with $\sum_{j=1}^m\lambda_j=\tr(\fou{\Op}_\um)\leq \tr(\Op)=\Ex\HnormV{X}^2$ implies \eqref{app:gauss:l2:e1}.\\

Proof of \eqref{app:gauss:l2:e2} and \eqref{app:gauss:l2:e3}.   Taking the inequality $\sum_{j=1}^m\lambda_j\leq \Ex\HnormV{X}^2$ and the identities
$n^4\rho_m^{-8}\normV{\fou{W}_\um}^8$ $= (\sum_{j=1}^m
\lambda_j(\sum_{i=1}^n\lU_i\lV_{ij})^2)^4$ and  $(\fou{W}_\um^t\fOp_\um^{-1}\fou{W}_\um)/\rho_m^2 =n^{-2} \sum_{j=1}^m (\sum_{i=1}^n\lU_i\lV_{ij})^2$   into account the
assertions \eqref{app:gauss:l2:e2} and \eqref{app:gauss:l2:e3} follow, respectively, from \eqref{app:gauss:l1:e5} and \eqref{app:gauss:l1:e2:1} in Lemma
\ref{app:gauss:l1} (with $a_j=\lambda_j$).\\

Proof of \eqref{app:gauss:l2:e4}. Since $n\mnormV{\fou{\Xi}_\um}\leq m\max_{1\leq j,l\leq m}|\sum_{i=1}^n(\lV_{ij}\lV_{il}-\delta_{jl})|$ we obtain due to 
\eqref{app:gauss:l1:e1} and \eqref{app:gauss:l1:e2} in Lemma \ref{app:gauss:l1}  for all $\eta>0$ the following bound 
\begin{multline*}
P(\mnormV{\fou{\Xi}_\um}\geq \eta)\leq \sum_{1\leq j,l\leq m}P(|n^{-1}\sum_{i=1}^n(\lV_{ij}\lV_{il}-\delta_{jl})|\geq \eta/m) \\ \leq m^2\max\set{
P(|n^{-1}\sum_{i=1}^n\lV_{i1}\lV_{i2}|\geq \eta/m),   P(|n^{-1/2}\sum_{i=1}^n(\lV_{i1}^2-1)|\geq n^{1/2}\eta/m)}\\ \leq m^2\max\set{(1+
\frac{m}{\eta n^{1/2}}) \exp\bigg(-\frac{n}{4}\min\set{\eta^2/m^2,1/4}\bigg),  2 \exp\bigg(-\frac{1}{8}\frac{n\eta^2/m^2}{1+\eta/m}\bigg)}.\end{multline*}
Moreover, for all $\eta\leq m/2$ this can be simplified to 
\[P(\mnormV{\fou{\Xi}_\um}\geq \eta)\leq  m^2\max\set{1+\frac{2m}{\eta n^{1/2}}, 2} \exp\bigg(-\frac{1}{12}\frac{n\eta^2}{m^2}\bigg),\]
which obviously implies \eqref{app:gauss:l1:e4}.\\

Proof of \eqref{app:gauss:l2:e5}.   Since  $Y_1/\sigma_Y,\dotsc,Y_n/\sigma_Y $ are independent and standard normally distributed, \eqref{app:gauss:l2:e5} follows from \eqref{app:gauss:l1:e1}  in Lemma
  \ref{app:gauss:l1} by exploiting that  $\{1/2\leq\hsigma^2_Y/\sigma^2_Y\leq3/2\}^c\subset \{|n^{-1}\sum_{i=1}^nY_i^2/\sigma_Y^2-1|>1/2\}$.\\

Proof of \eqref{app:gauss:l2:e6}. From the  identity  $n(\fou{W}_\um^t\fOp_\um^{-1}\fou{W}_\um)/(m\rho_m^2) = m^{-1}\sum_{j=1}^m ( n^{-1/2}\sum_{i=1}^n\lU_i\lV_{ij})^2$ the estimate  \eqref{app:gauss:l2:e6} follows by
using \eqref{app:gauss:l1:e5} in Lemma \ref{app:gauss:l1}, that is 
  \begin{multline*}
\sup_{m\geq1}\Ex\vectp{ \frac{n(\fou{W}_\um^t\fOp_\um^{-1}\fou{W}_\um)}{m\rho_m^2} -8 (1+\log n )}\leq  \Ex\vectp{ | n^{-1/2}\sum_{i=1}^n\lU_i\lV_{i1}|^2 -8 (1+\log n)}\\
\leq  \set{\frac{n^{-2}}{e^2\sqrt{\pi 2 (1+\log n)}}  + 64  \frac{(1+\log n)}{n}\exp(-{n}/{16})}\leq C n^{-2}.
  \end{multline*}

Proof of \eqref{app:gauss:l2:e7}. Define $\lV_i:= (\fLi_\um^t\fOp_\um^{-1}\fLi_\um)^{-1/2}\fLi_\um^t\fOp_\um^{-1}\fou{X_i}_\um$ for $1\leq i\leq n$, where  
  $\lU_1,\dotsc,\lU_n,$ $\lV_1,\dotsc,\lV_n$ are independent and standard normally distributed random variables. By employing the identity $n(\fLi_\um^t\fOp_\um^{-1}\fou{W}_\um)^2/(\rho^2_m\fLi_\um^t\fOp_\um^{-1}\fLi_\um)=
  |n^{-1/2}\sum_{i=1}^n\lU_i\lV_i|^2$  the estimate  \eqref{app:gauss:l2:e7} follows  from \eqref{app:gauss:l1:e5} in Lemma \ref{app:gauss:l1}, which completes the proof.
\end{proof}
\begin{lem}\label{app:gauss:l3} There exists a  constant $C(\Opd)$ only depending on $\Opd$ such that for all $n\geq 1$
  \begin{align}\label{app:gauss:l3:e1}
    &\sup_{\So\in\cSowr}\sup_{\Op\in\cOpwd}\sum_{ m=1}^{\Mon} \Ex\vectp{\frac{(\fLi_\um^t\fOp_{\um}^{-1}\fLi_\um)}{m}([W]_{\um}^t\fOp_{\um}^{-1}[W]_{\um}) -
      \frac{8\pen_m}{100}} \leq C(\Opd)(\sigma^2+\Sor) n^{-1};\\\label{app:gauss:l3:e2}
&    \sup_{\So\cSowr}\sup_{\Op\in\cOpwd}\sum_{ m=1}^{\Mon} \Ex\vectp{(\fLi_\um^t\fOp_{\um}^{-1}[W]_{\um})^2 -  \frac{8\pen_m}{100}} \leq C(\Opd)(\sigma^2+\Sor)  n^{-1}.
  \end{align}
\end{lem}
\begin{proof}[\noindent\textcolor{darkred}{\sc Proof of Lemma \ref{app:gauss:l3}.}]  
The key argument to show \eqref{app:gauss:l3:e1} is the estimate \eqref{app:gauss:l2:e6} in  Lemma~\ref{app:gauss:l2}. 
 Taking  $\fLi_\um^t\fOp_{\um}^{-1}\fLi_\um\leq V_m$ and $ \frac{8\pen_m}{100}= 8\,\sigma_m^2\,V_m\frac{1+\log n}{n}$ into account, together with the facts that $\max_{1\leq m\leq \Mon} V_m = V_{\Mon}\leq n C(\Opd)(1+\log n)^{-1}$   and  $\rho_m^2\leq \sigma_{m}^2\leq C(\Opd)(\sigma^2+\Sor)$ for all $\So\in\cSowr$, $\Op\in\cOpwd$ (Lemma \ref{app:pre:l1} (ii) and (iv)) we obtain 
\begin{multline*}
\sum_{ m=1}^{\Mon} \Ex\vectp{\frac{(\fLi_\um^t\fOp_{\um}^{-1}\fLi_\um)}{m}([W]_{\um}^t\fOp_{\um}^{-1}[W]_{\um}) -
      \frac{8\pen_m}{100}} \\
\hfill\leq  \sum_{ m=1}^{\Mon}\frac{\sigma_m^2V_m}{n}\, \Ex\vectp{\frac{n([W]_{\um}^t\fOp_{\um}^{-1}[W]_{\um})}{m\rho_m^2} - 8\,(1+\log n)}\hfill\\ 
\leq \frac{C(\Opd)(\sigma^2+\Sor)}{1+ \log n} \Mon  \sup_{m\geq 1} \Ex\vectp{\frac{([W]_{\um}^t\fOp_{\um}^{-1}[W]_{\um})}{m\rho_m^2} - 8\,(1+\log n)}.
\end{multline*}
The assertion \eqref{app:gauss:l3:e1} follows  by employing  \eqref{app:gauss:l2:e6} in  Lemma~\ref{app:gauss:l2} and $\Mon\leq n$. The proof of \eqref{app:gauss:l3:e2} follows the same lines by using  \eqref{app:gauss:l2:e7} in  Lemma~\ref{app:gauss:l2}  rather than   \eqref{app:gauss:l2:e6}  and we omit the details. \end{proof}
\begin{lem}\label{app:gauss:l4} There exists a numerical constant $C$ and a constant  $C(\Opd)$ only depending on $\Opd$ such that for all $n\geq 1$
\begin{align}\label{app:gauss:l4:e1}
\sup_{\So\in\cSowr}\sup_{\Op\in\cOpwd}& \set{n^4 (\Mon)^4 \max_{1\leq m\leq \Mon} P\vect{ \Xisetmn^c}}\leq C;\\\label{app:gauss:l4:e2}
\sup_{\So\in\cSowr}\sup_{\Op\in\cOpwd}&\set{n\,\Mon \max_{1\leq m\leq \Mon} P\vect{ \hOpsetmn^c}}\leq C(\Opd);\\\label{app:gauss:l4:e3}
\sup_{\So\in\cSowr}\sup_{\Op\in\cOpwd}&\set{n^7 P\vect{ \esetn^c}}\leq C.
\end{align}
\end{lem}
\begin{proof}[\noindent\textcolor{darkred}{\sc Proof of Lemma \ref{app:gauss:l4}.}] Since $\Mon\leq\gauss{n^{1/4}}$ and $\Xisetmn^c=\set{\sqrt{m}\mnormV{[\Xi]_{\um}}>1/8}$ the assertion \eqref{app:gauss:l4:e1}
  follows from \eqref{app:gauss:l2:e4} in Lemma \ref{app:gauss:l2}.\\ Consider \eqref{app:gauss:l4:e2}. With $ n_o:=n_o(\Opd):=\exp(128 \Opd^6)\geq 8\Opd^3$ we have $\normV{\fLi_{\underline{\Mon}}}^2(1+\log n)\geq 128\Opd^6$ for all $n\geq n_o$. We distinguish in the following the cases $n<n_o$ and $n\geq n_o$.  First, consider $1\leq n\leq n_o$.  Obviously, we have $\Mon\max_{1\leq m\leq \Mon} P(\hOpsetmn^c)\leq \Mon\leq n^{-1}
  n_o^{5/4}\leq C(\Opd)n^{-1}$ since $\Mon\leq n^{1/4}$ with $n_o$ depending on $\Opd$ only. On the other hand, if $n\geq n_o$ then  Lemma \ref{app:pre:l1} (iii) implies $n\geq 2\max_{1\leq m\leq
    \Mon}\mnormV{\fOp^{-1}_m}$, and hence $\Xisetmn\subset \hOpsetmn$ for all $1\leq m\leq \Mon$  by using Lemma \ref{app:pre:l4}. From  \eqref{app:gauss:l4:e1} we conclude $\Mon\max_{1\leq m\leq
    \Mon}P(\Omega^c_{m,n})\leq \Mon\max_{1\leq m\leq \Mon}P(\mho^c_{m,n})\leq C n^{-3}$. By combination of the two cases we obtain \eqref{app:gauss:l4:e2}.\\ It remains to show \eqref{app:gauss:l4:e3}. Consider the events $\asetn$, $\bsetn$ and $\csetn$ defined in \eqref{app:pre:e4}, where
  $\asetn\cap\bsetn\cap\csetn\subset\cE_n$ due to Lemma \ref{app:pre:l3}. Moreover, we have $n^7 P\vect{ \asetn^c}\leq C$ and $n^7 P\vect{ \csetn^c}\leq C$ due to \eqref{app:gauss:l2:e5} and \eqref{app:gauss:l2:e3}  in Lemma   \ref{app:gauss:l2}
(keep in mind that $\gauss{n^{1/4}}\geq\Mln$ and  $2(\sigma_Y^2+[\gf]_\uk^t\fOp_\uk^{-1}[\gf]_\uk)=\sigma_{k}^2\geq \rho_k^2$). Finally,
\eqref{app:gauss:l2:e4}  in Lemma   \ref{app:gauss:l2} implies $n^7 P\vect{ \bsetn^c}\leq C$
 by using that
$\{\mnormV{\sqrt{m}\fou{\Xi}_\um}\leq 1/8, 1\leq m\leq \Mon\}\subset\bsetn$. Combining these estimates yields \eqref{app:gauss:l4:e3}, which completes the proof.
\end{proof}
\subsection{Proof of Proposition \ref{adaptive:p1} and \ref{adaptive:p2}}\label{app:prop}
In the following proofs we will use the notations introduced in Appendix \ref{app:notations} and we will exploit the  technical assertions gathered in Lemma \ref{app:gauss:l1}-
\ref{app:gauss:l4}.
\begin{proof}[\noindent\textcolor{darkred}{\sc Proof of Proposition \ref{adaptive:p1}.}]
From the identities $\hLi_m-\Li(\So_m) =\fLi_\um^t\fhOp^{-1}_\um\fou{W}_{\um}\1_{\hOpsetmn} - \Li(\So_m)\1_{\hOpsetmn^c}$, $(\Id_\um+\fou{\Xi}_\um)^{-1}-\Id_\um= -(\Id_\um+\fou{\Xi}_\um)^{-1}\fou{\Xi}_\um$, and $\fhOp_{\um}=
\fOp^{1/2}_{\um}\{\Id_{\um}+[\Xi]_{\um}\}\fOp^{1/2}_{\um}$ follows
  \begin{multline*}
|\hLi_m-\Li(\So_m)|^2= |\fLi_\um^t\fhOp^{-1}_\um\fou{W}_{\um}|^2\1_{\hOpsetmn} + |\Li(\So_m)|^2\1_{\hOpsetmn^c}\\
\leq 2|\fLi_\um^t\fOp^{-1}_\um\fou{W}_{\um}|^2 + 2|\fLi_\um^t(\fhOp^{-1}_\um-\fOp_\um^{-1})\fou{W}_{\um}|^2\1_{\hOpsetmn} + |\Li(\So_m)|^2\1_{\hOpsetmn^c}\\
\leq 2|\fLi_\um^t\fOp^{-1}_\um\fou{W}_{\um}|^2 + 2|\fLi_\um^t\fOp^{-1/2}_\um(\Id_\um+\fou{\Xi}_\um)^{-1}\fou{\Xi}_\um\fOp_\um^{-1/2}\fou{W}_{\um}|^2\1_{\Xisetmn}\\
+ 2|\fLi_\um^t\fOp_\um^{-1/2}\fou{\Xi}_\um\fOp_\um^{1/2}\fhOp^{-1}_\um\fou{W}_{\um}|^2\1_{\hOpsetmn}\1_{\Xisetmn^c}+ |\Li(\So_m)|^2\1_{\hOpsetmn^c}. 
\end{multline*}
By exploiting  $\sqrt{m}\mnormV{(\Id_{\um}+[\Xi]_{\um})^{-1}[\Xi]_{\um}}\1_{\Xisetmn}\leq 1/7$ and $\mnormV{\fhOp^{-1}_\um}\1_{\hOpsetmn}\leq n$ we obtain
  \begin{multline*}
|\hLi_m-\Li(\So_m)|^2
\leq 2|\fLi_\um^t\fOp^{-1}_\um\fou{W}_{\um}|^2 + \frac{2}{49}(\fLi_\um^t\fOp^{-1}_\um\fLi_\um)  m^{-1} (\fou{W}_{\um}^t\fOp_\um^{-1}\fou{W}_{\um})\\
+ 2n^2\,(\fLi_\um^t\fOp_\um^{-1}\fLi_\um)\, \mnormV{\fou{\Xi}_\um\fOp_\um^{1/2}}^2\, \normV{\fou{W}_{\um}}^2\1_{\Xisetmn^c}+ |\Li(\So_m)|^2\1_{\hOpsetmn^c}. 
\end{multline*}
Taking  this upper bound into account together with $(\fLi_\um^t\fOp_\um^{-1}\fLi_\um)\leq V_m$, we obtain for all $\So\in\cSowr$ and $\Op\in\cOpwd$ that
\begin{multline*}
  \Ex\set{\sup_{1\leq m\leq \Mon} \vectp{|\hLi_m-\Li(\So_m)|^2-\frac{1}{6}\pen_m}}
\leq  2\sum_{ m=1}^{\Mon}\Ex\vectp{|\fLi_\um^t\fOp^{-1}_\um\fou{W}_{\um}|^2- \frac{8}{100}\pen_m}\\\hfill+\frac{2}{49}\sum_{ m=1}^{\Mon}\Ex\vectp{(\fLi_\um^t\fOp^{-1}_\um\fLi_\um^t)  m^{-1} (\fou{W}_{\um}^t\fOp_\um^{-1}\fou{W}_{\um})- \frac{8}{100}\pen_m}
\\+2n^3 \sum_{m=1}^{\Mon}  \frac{V_m}{n} \big(\Ex\mnormV{\fou{\Xi}_\um\fOp_\um^{1/2}}^8\big)^{1/4} \big(\Ex\normV{[W]_{\um}}^8\big)^{1/4}\big(P(\Xisetmn^c)\big)^{1/2}  +\sum_{m=1}^{\Mon} |\Li(\So_m)|^2 P(\hOpsetmn^c).
\end{multline*}
We bound  the first and second right hand side term with help of \eqref{app:gauss:l3:e1} and \eqref{app:gauss:l3:e2} in Lemma \ref{app:gauss:l3}, which leads to
\begin{multline*}
\sup_{\So\in\cSowr}\sup_{\Op\in\cOpwd} \Ex\set{\sup_{1\leq m\leq \Mon} \vectp{|\hLi_m-\Li(\So_m)|^2-\frac{1}{6}\pen_m}}
\leq  C(\Opd)(\sigma^2+\Sor)n^{-1}\\
\hfill+2n^3 \sup_{\So\in\cSowr}\sup_{\Op\in\cOpwd} \sum_{m=1}^{\Mon}  \frac{V_m}{n} \big(\Ex\mnormV{\fou{\Xi}_\um\fOp_\um^{1/2}}^8\big)^{1/4}
\big(\Ex\normV{[W]_{\um}}^8\big)^{1/4}\big(P(\Xisetmn^c)\big)^{1/2}\\  + \sup_{\So\in\cSowr}\sup_{\Op\in\cOpwd}\sum_{m=1}^{\Mon} |\Li(\So_m)|^2 P(\hOpsetmn^c).
\end{multline*}
 Taking into account that   for all $\So\in\cSowr$ and $\Op\in\cOpwd$ we have $\max_{1\leq m\leq \Mon} V_m = V_{\Mon}\leq n C(\Opd)(1+\log n)^{-1}$   and  $\rho_m^2\leq
 \sigma_{m}^2\leq C(\Opd)(\sigma^2+\Sor)$  (Lemma \ref{app:pre:l1} (ii) and (iv)) the estimates \eqref{app:gauss:l2:e1} and \eqref{app:gauss:l2:e2} in Lemma \ref{app:gauss:l2} imply
\begin{multline*}
\sup_{\So\in\cSowr}\sup_{\Op\in\cOpwd} \Ex\set{\sup_{1\leq m\leq \Mon} \vectp{|\hLi_m-\Li(\So_m)|^2-\frac{1}{6}\pen_m}}
\leq  \frac{C(\Opd)}{n}(\sigma^2+\Sor)\\
\hfill+\frac{C(\Opd)}{n}(\sigma^2+\Sor) \sup_{\So\in\cSowr}\sup_{\Op\in\cOpwd} (\Ex\HnormV{X}^2)^2   n^2(\Mon)^2 \max_{1\leq m\leq \Mon}\big(P(\Xisetmn^c)\big)^{1/2}\\  + \sup_{\So\in\cSowr}\sup_{\Op\in\cOpwd}\sum_{m=1}^{\Mon} |\Li(\So_m)|^2 P(\hOpsetmn^c).
\end{multline*}
By combining this upper bound, the property $\Ex\HnormV{X}^2\leq \Opd \sum_{j\geq 1}\Opw_j$ and the estimate \eqref{minimax:l1:e5} given in Lemma \ref{minimax:l1} we obtain 
\begin{multline*}
\sup_{\So\in\cSowr}\sup_{\Op\in\cOpwd} \Ex\set{\sup_{1\leq m\leq \Mon} \vectp{|\hLi_m-\Li(\So_m)|^2-\frac{1}{6}\pen_m}}
\leq  \frac{C(\Opd)}{n}(\sigma^2+\Sor)\\
\hfill+\frac{C(\Opd)}{n}(\sigma^2+\Sor) (\sum_{j\geq 1}\Opw_j)^2 \sup_{\So\in\cSowr}\sup_{\Op\in\cOpwd}    n^2(\Mon)^2 \max_{1\leq m\leq \Mon}\big(P(\Xisetmn^c)\big)^{1/2}\\  +
\frac{C(\Opd)}{n}\Sor \sum_{j\geq1}\frac{\fLi^2_j}{\Sow_j} \sup_{\So\in\cSowr}\sup_{\Op\in\cOpwd} n\Mon \max_{1\leq m\leq \Mon}P(\hOpsetmn^c).
\end{multline*}
 The result of the proposition follows now from the  upper bounds \eqref{app:gauss:l4:e1} and \eqref{app:gauss:l4:e2} given in Lemma \ref{app:gauss:l4}, which completes the proof.
\end{proof}
\begin{proof}[\noindent\textcolor{darkred}{\sc Proof of Proposition \ref{adaptive:p2}.}]
Taking the estimate $\mnormV{\fhOp^{-1}_\um}\1_{\hOpsetmn}\leq n$  and the identity $\hLi_m-\Li(\So_m)\1_{\hOpsetmn}
=\fLi_\um^t\fhOp^{-1}_\um\fou{W}_{\um}\1_{\hOpsetmn}$ into account    it  easily follows for all $m\geq 1$  that 
 \begin{equation*}
     \absV{\hLi_{m}-\Li(\So)}^2\leq 3\{\normV{\fLi_\um}^2\,n^2\normV{\fou{W}_{\um}}^2  + (|\Li(\So_{m})|^2+|\Li(\So)|^2)\}.
    \end{equation*}
 Furthermore, by exploiting $\normV{\fLi_\um}^2\leq n$ for all $1\leq m \leq \Mln$ we obtain from the last estimate
 \begin{equation*}
    \max_{1\leq m\leq \Mln}\absV{\hLi_{m}-\Li(\So)}^2\1_{\esetn^c}\leq  3 \{n^3 \sum_{m=1}^{\Mln}\normV{\fou{W}_{\um}}^2\1_{\esetn^c} +(\sup_{m\geq 1}|\Li(\So_{m})|^2+|\Li(\So)|^2)\1_{\esetn^c}\}.
    \end{equation*}
We recall that   for all $\So\in\cSowr$ and $\Op\in\cOpwd$ we have $\rho_m^2\leq C(\Opd)(\sigma^2+\Sor)$  and $(\Ex\normV{\fou{W}_{\um}}^4)^{1/2}\leq 11 \Ex\HnormV{X}^2
 \rho_m^2 n^{-1}$   (Lemma  \ref{app:pre:l1} and \ref{app:gauss:l2}), moreover,  the bounds
$\big(\sup_{m\geq 1}|\Li(\So_{m})|^2+|\Li(\So)|^2\big)\leq (\sup_{m\geq 1}\normV{\So_{m}}^2_\Sow+
 \normV{\So}_\Sow^2)\sum_{j\geq1}\frac{\fLi_j^2}{\Sow_j}\leq C(\Opd)\Sor\sum_{j\geq1}\frac{\fLi_j^2}{\Sow_j} $ (Lemma \ref{minimax:l1}) and  $\Ex\HnormV{X}^2\leq \Opd\sum_{j\geq
   1}\Opw_j$  together with the last upper bound imply
\begin{multline*}
\sup_{\So\in\cSowr}\sup_{\Op\in\cOpwd}\Ex\big(\absV{\hLi_{\whm}-\Li(\So)}^2\1_{\esetn^c}\big)\leq \sup_{\So\in\cSowr}\sup_{\Op\in\cOpwd}\Ex\big( \max_{1\leq m\leq \Mln}\absV{\hLi_{m}-\Li(\So)}^2\1_{\esetn^c}\big)\\\leq
 C(\Opd)\,(\sigma^2+\Sor)\max\set{ \sum_{j\geq1}\Opw_j, \sum_{j\geq1}\frac{\fLi_j^2}{\Sow_j}}  \sup_{\So\in\cSowr}\sup_{\Op\in\cOpwd} \vect{ n^2\Mln |P(\esetn^c)|^{1/2}+ 
 P(\esetn^c)}.
\end{multline*}
The assertion of Proposition \ref{adaptive:p2} follows now by combination of the last estimate and \eqref{app:gauss:l4:e3} in Lemma \ref{app:gauss:l4}, which completes the proof.
\end{proof}
\bibliography{2011-12-09-FLM-LinFun-Adapt-JJRS}
\end{document}